%% file: main.tex
\documentclass[a4paper, nonatbib, 3p, twocolumn]{elsarticle}
\input{header}

\DeclareUnicodeCharacter{2009}{~}

\begin{document}
\begin{frontmatter}
\title{Koopman-based stability analysis of differential-algebraic equations with applications to frictional multibody systems}
\author[unistutt]{Alexander Schütz}
\ead{schuetz@inm.uni-stuttgart.de}
\author[unistutt]{Fabia Bayer\corref{cor1}}
\ead{bayer@inm.uni-stuttgart.de}
\author[unistutt]{Remco I. Leine}
\ead{leine@inm.uni-stuttgart.de}
\affiliation[unistutt]{organization={University of Stuttgart, Institute for Nonlinear Mechanics},addressline={Pfaffenwaldring 9},city={70569 Stuttgart},country={Germany}}
\cortext[cor1]{Corresponding author}

\begin{abstract}
    Periodic solutions of differential-algebraic equations (DAEs) and ordinary differential equations (ODEs) can be determined using the harmonic balance method (HBM), which is a frequency-domain approach that approximates the solution by its truncated Fourier series. The Koopman-Hill method, a method to determine the stability of periodic solutions found by HBM, and originally developed for ODEs, is generalized to DAEs in this work. Analogously to the ODE case, the core idea of the proposed Koopman-Hill method for DAEs is to establish a linear time-invariant but high-dimensional DAE which approximately governs the dynamics of small admissible perturbations around the periodic solution. The crucial difference to the ODE case is the fact that the evolution of this linear time-invariant DAE is not simply given by a matrix exponential, but by a more complicated expression involving a Drazin inverse, rendering the resulting monodromy matrix singular. Still, even in the DAE case, this novel relationship between the monodromy matrix and the Hill matrix is essentially given by one single formula, which is the main result of this work. 
    Two academic mechanical systems, a mathematical pendulum formulated as an index-3 DAE and a nonsmooth frictional two-mass oscillator with switching index, demonstrate the applicability of the proposed method and its blindness to the DAE's index.

\end{abstract}

\begin{keyword}
    Coulomb friction \sep Floquet theory \sep Harmonic Balance Method \sep Koopman theory \sep multibody systems \sep nonsmooth dynamics \sep stability analysis
\end{keyword}
\end{frontmatter}

\section{Introduction}\label{ch:Introduction}
In this work, we propose a novel method for the stability analysis of periodic solutions of nonlinear differential algebraic equations. In essence, the derived approach comprises the computation of the monodromy matrix for such periodic solutions by solving a lifted linear time-invariant differential-algebraic equation (DAE) that naturally arises as a by-product of the harmonic balance method.

DAEs occur in many engineering and scientific applications. Examples include bilateral constraints in multibody dynamics \cite{Roberson1988}, incompressible fluid flow~\cite{Harlow1965, Montlaur2011} and electrical circuits with Kirchhoff constraints~\cite{Gunther1995, Scholz2018}. Closed-form solutions of DAEs are available only in special cases \cite{Campbell1976,Trenn2013}. Consequently, much research has been dedicated to the numerical integration of DAEs \cite{Brenan1995,Gear1985,Anantharaman1991,Trenn2013}. These time-domain integrators usually require index reduction to guarantee accuracy and stability.

In the area of nonsmooth dynamics~\cite{Leine2008,Acary2008}, DAEs result naturally when implicit constitutive laws are present, e.g., when modeling set-valued frictional contact laws in nonsmooth mechanics \cite{Glocker2001}. 
Herein, the proximal point function \cite{Leine2008} plays a crucial role as it allows to transform set-valued normal cone inclusions into implicit algebraic equations. Using this methodology, mechanical systems with Coulomb friction, originally formulated as set-valued differential inclusions, may be reformulated as DAEs. However, the switching nature of the set-valued Coulomb friction law causes the index of the DAE to change at slip-stick and stick-slip transitions. This renders standard time-domain integration methods using index-reduction inapplicable and gives rise to dedicated integrators such as the Moreau timestepping scheme~\cite{Acary2008} or the nonsmooth generalized-alpha methods~\cite{Capobianco2021}.

The harmonic balance method  (HBM) is a widely-used frequency-domain method to identify periodic solutions of ordinary differential equations \cite{Krack2019} with applications, for instance, in turbomachinery and rotordynamics~\cite{VonGroll2001}. The HBM does not rely on numerical time-integration. It is usually implemented using an alternating frequency and time approach, where nonlinearities are evaluated in the time domain and fast-Fourier-transformed to the frequency domain, where the Fourier coefficients of the ODE are balanced in a Galerkin procedure.

For ODEs, a broad variety of stability analysis tools for periodic solutions is available \cite{LaSalle1968, Verhulst2006, Peletan2013}. A classical frequency-domain approach for the stability analysis, which is especially suited for HBM contexts, is the Hill eigenvalue problem\cite{Lazarus2010,Zhou2004}. It hinges on the complete eigenvalue decomposition of the Jacobian matrix of the HBM residual, which is known as the truncated Hill matrix. Despite being widely used, a major drawback of this method is the necessity of a sorting algorithm to remove unphysical, ``spurious'' eigenvalues \cite{Wu2022}. The recently proposed Koopman-Hill method for ODEs~\cite{Bayer2023,Bayer2024a,Bayer2025} replaces this eigenvalue decomposition by the direct computation of the monodromy matrix through a matrix exponential. This matrix exponential arises from a re-interpretation of the Hill matrix as the system matrix of a high-dimensional linear time-invariant (LTI) ODE.

The HBM can, in principle, also be applied to DAEs. While HBM application to electrical circuits has been well established~\cite{Gilmore1991}, the HBM has been applied to mechanical DAEs only more recently~\cite{Ju2021, Cochelin2009}. As the method does not require an index reduction, it is particularly suitable for the analysis of DAEs of high or switching index such as those arising in frictional mechanical multibody systems, as recently exemplified by a nonsmooth frictional oscillator~\cite{Legrand2024,Hashemi2025}. While the HBM as presented in~\cite{Legrand2024} allows for the identification of periodic solutions in frictional multibody systems, a frequency-based stability method for DAEs, which complements the HBM for DAEs, is missing. In contrast to the rich theory developed for ODEs, stability analysis tools for periodic solutions in DAEs are only available in the time domain, restricted to DAEs with constant index 1 or 2~\cite{Lamour1998, Lamour2003}. Therefore, the motivation of the present work is to supplement the results of~\cite{Legrand2024} with an accurate HBM-based stability analysis tool for DAEs of arbitrary index.

The central contribution of the present work is a novel approach for the stability analysis of periodic solutions of DAEs, applicable in conjunction with HBM. To this end, we generalize the Koopman-Hill method~of \cite{Bayer2023,Bayer2024a,Bayer2025} to DAEs. While the core idea of the proposed Koopman-Hill method for DAEs is analogous to the ODE case, the crucial difference is the treatment of the resulting linear time-invariant DAE, which is not simply solved by a matrix exponential. Instead, the solution involves a more complicated expression with a Drazin inverse, which renders the resulting monodromy matrix singular and can be numerically challenging. Despite being slightly more involved than in the ODE case, the Koopman-Hill relationship between the monodromy matrix and the Hill matrix of a DAE is still given by one single formula~\eqref{eq:abstract:mainFormula}, which is the main result of this work. Notably, the proposed formula is independent of the DAE's index, which makes it particularly suitable for the analysis of frictional mechanical multibody systems with switching index due to Coulomb friction.

The paper is structured as follows. Section \ref{ch:Multibody} discusses the considered DAE system class and demonstrates how such DAEs naturally arise as models for mechanical multibody systems with smooth bilateral or frictional constraints.  
Section \ref{ch:HBM} provides a brief overview over the HBM as a frequency-domain approach for the identification of periodic solutions in ODEs and DAEs. Our main contribution, the generalization of the Koopman-Hill method to DAEs, is presented in Section \ref{ch:KHDAE}. After highlighting relevant aspects of Floquet theory for DAEs, the linear time-invariant Hill DAE is derived in analogy to the ODE case. With the Hill DAE at hand, solution properties of linear time-invariant DAEs are used to derive the monodromy matrix through the Koopman-Hill equation~\eqref{eq:abstract:mainFormula}. 

The accuracy and the numerical properties of the proposed method are analyzed in Section \ref{ch:Examples} for two example systems whose dynamics can be formulated as a DAE: a nonlinear mathematical pendulum formulated as a point mass with a bilateral position constraint and a nonsmooth frictional oscillator. For the pendulum example, the results obtained by the proposed Koopman-Hill method for DAEs are verified against the results obtained from an equivalent ODE formulation of the system. For the frictional oscillator, the results of the Koopman-Hill method applied to the nonsmooth DAE as well as to a regularized ODE formulation are compared to a highly accurate time-domain Floquet analysis. Concluding remarks are made in Section~\ref{ch:Conclusion}.

\section{DAEs modeling frictional multibody systems}\label{ch:Multibody}

Mechanical multibody systems, derived for example using Lagrange's formalism or the projected Newton-Euler equations, are usually given in a structured second-order form~\cite{GarciadeJalon1994}.
As a starting point, consider a constrained mechanical multibody system described by the DAE

\begin{subequations}\label{eq:Multibody:Lagrange1}
\begin{align}
    \vM\ddot{\vq}-\vh(t,\vq, \dot{\vq}) &= \vW(t, \vq)\vla  \label{eq:Multibody:Lagrange1:ode} \\
    \vzero &= \vg(t, \vq, \dot{\vq}, \vla) \;.
\end{align}
\end{subequations}
Here $\vq(t) \in \Rspace^f$ denotes generalized coordinates and $\dot{\vq}(t) \in \Rspace^f$ the corresponding generalized velocities.
The vector $\vh$ collects internal and external forces. This includes potential and non-potential contributions such as gravity, gyroscopic terms, damping, and parametric or external excitation. 
The positive definite mass matrix $\vM$ is assumed here to be constant. Systems with a state-dependent mass matrix can be written in the form~\eqref{eq:Multibody:Lagrange1} by a suitable coordinate transform or by pre-multiplying Equation~\eqref{eq:Multibody:Lagrange1:ode} with~$\vM^{-1}$.

The mechanical system~\eqref{eq:Multibody:Lagrange1} is subject to~$c$ constraint equations collected in $\vg(t, \vq, \dot{\vq}, \vla) \in \Rspace^{c}$. The corresponding constraint forces $\vla(t)\in\Rspace^c$ enter Equation~\eqref{eq:Multibody:Lagrange1:ode} as generalized forces with their directions being the columns of $\vW(\vq)$ \cite{Kane1985}. We assume these directions are linearly independent, i.e., the constraints $\vg$ are non-redundant.
Position-level constraints of the form $0=\vg_{\mathrm{pos}}(t,\vq)$ have generalized force directions
\begin{align}
    \vW_{\mathrm{pos}}(t,\vq) = \left(\pd{\vg_{\mathrm{pos}}}{\vq}\right)\T
\end{align}
 and yield an index-3 DAE~\cite{Lamour2003}.
Velocity level constraints of the form
\begin{align}
    0 = \vg_{\mathrm{vel}}(t, \vq, \dot{\vq}) = \vW_{\mathrm{vel}}(t, \vq)\T \dot{\vq} + \vch(t, \vq)
\end{align}
with generalized force directions 
\begin{align}
    \vW_{\mathrm{vel}}(t,\vq) = \left(\pd{\vg_{\mathrm{vel}}}{\dot{\vq}}\right)\T
\end{align}
lead to a DAE of index 2.
Finally, force level constraints
\begin{align}
    0 = \vg_{\mathrm{force}}(t, \vq, \dot{\vq}, \vla)
\end{align}
with nonsingular $\pd{\vg_{\mathrm{force}}}{\vla}$ lead to an index-1 DAE. 

Planar unilateral frictional contact between two bodies with relative sliding velocity $\gamma_T$ induces a scalar contact force $\lambda_N \geq 0$ in normal direction and a scalar tangential friction force $\lambda_T$ \cite{Glocker2001}.
Dry friction distinguishes sliding phases, where $\lambda_T$ admits a prescribed value, and sticking phases, where $\lambda_T$ is a reactive force.
This behavior is, for instance, imposed by the nonsmooth set-valued Coulomb force law
\begin{align}\label{eq:Multibody:Inclusionlaw}
    - \lambda_{T} \in \mu\lambda_{N}\Sgn(\gamma_{T})
\end{align}
with friction coefficient $\mu > 0$, which will be considered in this work.
The set-valued sign function $\Sgn$ admits the values $+1$ and $-1$ when $\gamma_T$ is strictly positive or negative, respectively (sliding behavior).
When the relative velocity is exactly $0$ (sticking behavior), the sign function maps to the set $[-1, 1]$ and the friction force $-\lambda_T$ may admit any value in the interval $[-\mu \lambda_N, \mu \lambda_N]$, balancing internal and external forces such that the sticking condition is retained. The set-valued Coulomb force law is shown in black in Figure~\ref{fig:mb:forcelaw}. 

\begin{figure}[hbt]
	\centering
    \includegraphics{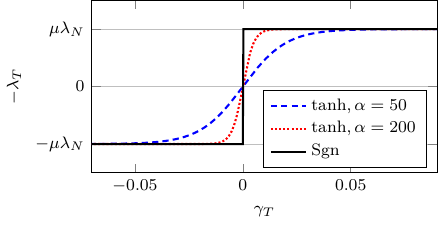}
    \caption{Comparison of nonsmooth Coulomb friction law with smoothed approximation.}
    \label{fig:mb:forcelaw}
\end{figure}

The force law~\eqref{eq:Multibody:Inclusionlaw} is formulated as an inclusion, which renders it unsuitable for being incorporated into the constraint equations of the DAE~\eqref{eq:Multibody:Lagrange1}.
A straightforward way to avoid the inclusion is to regularize the force law, for example with the approximation
\begin{align}
    \label{eq:Multibody:Inclusionlaw:smoothed}
    -\lambda_{T} = \mu\lambda_{N}\tanh(\alpha \gamma_{T}) \;,
\end{align}
where $\alpha$ is a nonphysical regularization parameter which controls the slope of the function at the origin. The regularized force law for two choices of $\alpha$ is shown in Figure~\ref{fig:mb:forcelaw}. 
The regularization procedure, no matter how steep the slope is chosen, removes qualitative properties from the solution.
For instance, a sticking phase of a frictional system can only be represented by the regularized force law as small oscillations with nonzero velocity.
In addition, large values of $\alpha$ render the considered system stiff and hard to solve numerically.

To avoid these issues of the smoothed approximation, it is possible to directly transform the inclusion~\eqref{eq:Multibody:Inclusionlaw} into an algebraic equation without loss of its nonsmooth properties by using tools from convex analysis. 
The procedure for general, more intricate inclusions, such as in the case of spatial Coulomb friction, involves the use of normal cone inclusions and the proximal point operator~\cite{Glocker2001,Leine2008}. For the sake of conciseness, we refrain here from a detailed derivation and, instead, verify the following algebraic equation using a case distinction.

The solution set of the equation
\begin{equation}\label{eq:Multibody:FrictionProx}
    \begin{aligned}
	0 = \gamma_{T} &+ \min\left(0,\rho(\lambda_{T}+\mu\lambda_{N}) - \gamma_{T}\right) \\
    &+ \max\left(0,\rho(\lambda_{T}-\mu\lambda_{N}) - \gamma_{T}\right)
    \end{aligned}
\end{equation}
with arbitrary $\rho>0$ consists exactly of the pairs $(\gamma_T, \lambda_T)$ which admit the inclusion~\eqref{eq:Multibody:Inclusionlaw}.
Generally, at most one of the $\min$ and $\max$ terms of~\eqref{eq:Multibody:FrictionProx} is nonzero at once for positive $\lambda_N$. For  positive tangential velocities $\gamma_T$, the $\min$ term must admit the value $-\gamma_T$, ensuring that the friction force satisfies $\lambda_T = -\mu \lambda_N$. Analogously, for negative tangential velocities, $\lambda_T = \mu \lambda_N$ is enforced through the $\max$ term. When the tangential velocity is exactly zero, both $\min$ and $\max$ terms vanish, imposing that $\lambda_T$ lies within the force reservoir $[-\mu \lambda_N, \mu \lambda_N]$. 

Equation~\eqref{eq:Multibody:FrictionProx} is an algebraic equation exactly representing the Coulomb friction law. Thus, for each frictional contact of a mechanical multibody system, one simply adds Equation~\eqref{eq:Multibody:FrictionProx} as an additional constraint to the DAE~\eqref{eq:Multibody:Lagrange1}.
DAEs with frictional constraints have a switching index: While the friction force is an impressed force during slipping phases, which corresponds to an index of one, a sticking phase behaves like a bilateral constraint on velocity level, yielding an index-2 DAE.

Equation~\eqref{eq:Multibody:Lagrange1}, whether frictional constraints are included or not, can be brought into first-order form by collecting its states in the state vector $\vx\T = [\vq\T, \dot{\vq}\T, \vla\T]$ of length $n = 2f + c$. The resulting DAE is of the form
\begin{align}\label{eq:Multibody:nonlinDAE}
    \vA \dot{\vx} = \vf(t, \vx) \;,
\end{align}
with the block-diagonal, constant pseudo-mass matrix
\begin{align}
    \vA = \begin{bmatrix}
    \vI & \vzero & \vzero \\ \vzero & \vM & \vzero \\ \vzero & \vzero & \vzero
    \end{bmatrix} \in \Rspace^{n\times n}\textnormal{,}
\end{align}
which is singular when constraints are present, and the right-hand side
\begin{align}
    \vf(t, \vx) = \begin{bmatrix}
    \dot{\vq} \\ 
    \vh(t, \vq, \dot{\vq}) + \vW(t, \vq) \vla\\ 
    \vg(t, \vq, \dot{\vq}, \vla) \\ 
    \end{bmatrix}\;.
\end{align}
As the mechanical mass matrix $\vM$ is invertible, the rank deficiency of the pseudo-mass matrix $\vA$ is due to the all-zero lower rows which come from the constraint equations. More quantitatively, the rank of~$\vA$ is $2f = n-c$. 

The main goal of this work is to analyze the stability properties of periodic solutions in systems of the form~\eqref{eq:Multibody:nonlinDAE}.
Hence, we are primarily interested in dynamics whose right hand side satisfies the periodicity property
\begin{equation}
    \vf(t, \vx) = \vf(t+T, \vx), \quad\mathrm{for~all~} t\in\Rspace, \forall \vx\in\Rspace^n
\end{equation}
 with period $T$.
This is, for instance, true for mechanical multibody systems with harmonic excitation.

We further assume that the solutions of~\eqref{eq:Multibody:nonlinDAE} are unique in forward time, an assumption which is justified for multibody systems with smooth bilateral constraints~\cite{Jalon1994} and Coulomb friction with constant normal force $\lambda_N$~\cite{Leine2004}. Uniqueness in backward time is, however, generally not guaranteed if nonsmooth Coulomb friction forces are present \cite{Leine2004}.

Motivated by the application to constrained frictional mechanical multibody systems, the investigation of the dynamics \eqref{eq:Multibody:nonlinDAE} with singular pseudo-mass matrix $\vA$ is of primary interest in the subsequent analysis.
However, the derived methods are generally also applicable in the case of regular $\vA$.
In particular, for the special case $\vA=\vI$, the considered system class~\eqref{eq:Multibody:nonlinDAE} includes ordinary differential equations (ODEs) in standard first-order form,
\begin{equation}\label{eq:Multibody:nonlinODE}
    \dot{\vx} = \vf(t, \vx)\;.
\end{equation}

\section{Harmonic balance for DAEs}\label{ch:HBM}
The main focus of the present work lies on the Koopman-Hill stability analysis of periodic solutions of the DAE~\eqref{eq:Multibody:nonlinDAE}, motivated by potentially frictional multibody systems.
Periodic solutions can be computed numerically through the harmonic balance method (HBM). While mainly being a tool for ODEs~\cite{Krack2019,Nayfeh1995}, the HBM can also be applied to DAEs of the form~\eqref{eq:Multibody:nonlinDAE}~\cite{Legrand2024}, as discussed in the following section.

A periodic solution $\vx^{\rp}$ with period~$T$ is a solution which solves the $T$-periodic DAE~\eqref{eq:Multibody:nonlinDAE} and  fulfills $\vx^\rp(t + T) = \vx^\rp(t)$ for all $t \geq 0$. 
By not restricting $T$ to be the minimal period of $\vx^{\rp}$ or $\vf$, the setup considered here includes, without loss of generality, subharmonic and superharmonic cases, where the ratio between the minimal periods of the solution and the system is a rational number. 
Trivially, this formulation also encompasses equilibrium points $\vx^{\rp}(t) \equiv \vx^* $ as periodic solutions with arbitrary period~$T$.

The task of finding $T$-periodic solutions of~\eqref{eq:Multibody:nonlinDAE} constitutes a boundary value problem (BVP)
as conditions are imposed on both the initial and terminal states. 
Various numerical methods exist to address this type of BVP, both in the time domain and in the frequency domain. 
Time-domain techniques such as shooting, multiple shooting, and collocation all rely on an interplay between time integration (or finite differencing) of the DAE using a numerical DAE solver and the solution of nonlinear equations for periodicity and continuity constraints~\cite{Brenan1995}. Hereby, the fact that the original dynamics is described as a DAE necessitates the use of specialized solvers. In particular, for DAEs of index higher than one, an index reduction and stabilization is usually required to ensure accurate integration results~\cite{Gear1985, Anantharaman1991}. 

In contrast, the HBM is a frequency-based method. It aims to determine a set of Fourier coefficients that approximate the periodic solution and is widely used for ODEs due to its efficiency and simplicity~\cite{Detroux2015,Kerschen2009,Krack2019,Lazarus2010,Peletan2013}. Recently, the HBM has been successfully applied to DAEs of the form~\eqref{eq:Multibody:nonlinDAE} with nonsmooth frictional constraints~\cite{Legrand2024,Hashemi2025}. This idea, which is agnostic to the DAE index and, thus, does not require index reduction, is summarized below.

The truncated Fourier analysis operator $\cF$ denotes the map from a continuous-time vector-valued $T$-periodic signal $\vx^{\rp}: [0, T) \rightarrow \Cspace^n$ with maximum bandwidth $N$ to its Fourier coefficient vector $\vX := \left({\vX_{-N}}\T, \dots, {\vX_N }\T\right)\T {\in \Cspace^{n (2N+1)}}$, which collects the Fourier coefficients up to the $N$-th harmonic. Its inverse, the Fourier synthesis operator $\cF^{-1}$,  reconstructs the signal from its Fourier coefficients via the Fourier series,
either in the complex formulation
\begin{align}\label{eq:def_xp:cplx}
	\vx^\rp(t) = \sum_{k = -N}^{N} \vX_k \ex^{\ic k \omega t} =: \cF^{-1}(\vX,t)
\end{align}
or in the equivalent real formulation
\begin{align}\label{eq:def_xp:real}
	\vx^\rp(t) = \vX_0 + \sum_{k = 1}^{N} \vX_{\rc, k} \cos(k \omega t) + \vX_{\rs, k} \sin(k \omega t)
\end{align}
with the angular frequency $\omega = \frac{2\pi}{T}$. 
Both formulations are equivalent and can be transformed into each other by a linear transformation~\cite{Bayer2024}. While the real-valued form is preferable for numerical implementation because it avoids complex arithmetic, the complex-valued form is more concise for theoretical analysis and will be used in the arguments below.

The HBM aims to numerically approximate a periodic solution of~\eqref{eq:Multibody:nonlinDAE} through a Fourier coefficient vector $\vX$ that contains the solution's Fourier coefficients.
Naturally, the $2N+1$ Fourier coefficients in~$\vX$ cannot represent all periodic functions exactly. Indeed, $\cF \cF^{-1}$ is the identity operator on $\Cspace^{n(2N+1)}$, while $\cF^{-1} \cF$ is an orthogonal projection from the infinite-dimensional space of $T$-periodic vector-valued functions onto the (vector-valued) span of the finitely many basis functions $\ex^{-\ic N \omega t}, \dots, \ex^{\ic N \omega t}$.
Nonetheless, the a priori unknown periodic solution $\vx^{\rp}(t)$ can be approximately parameterized by finitely many Fourier coefficients $\vX_{-N}, \dots, \vX_{N}$,
constrained by the dynamics
\begin{align}\label{eq:background:residu}
	\vr(\vX, t) = \vf(t, \cF^{-1}(\vX, t)) - \vA \td{}{t} \left(\cF ^{-1}(\vX, t)\right) \overset{!}{=}\vzero \;.
\end{align}
Ideally, if $\vX$ would describe the periodic solution exactly, then the residual $\vr$ would vanish at every time instant. However, as the periodic solution is generally only approximately described by $\vX$ due to the finite truncation, this will not be the case. Therefore, in the HBM, only the first $N$ harmonics of the periodic function~$\vr$ are required to vanish in a Galerkin procedure. This procedure
yields the $n(2N+1)$ algebraic equations
\begin{equation}
\label{eq:hbm:Rofx}
\vR(\vX) = \cF \vf\left(t, \cF^{-1}(\vX, t)\right) - \omega (\vD \otimes \vA) \vX = \vzero\;,
\end{equation}
where $\vD := \diag(-\ic N, \dots, \ic N) \otimes \vI$ is the differentiation operator. 
Equation~\eqref{eq:hbm:Rofx} is referred to as the harmonic balance equations.
The harmonic balance equations are often equivalently introduced with a negated sign, where the terms related to~$\vf$ are subtracted from the $\omega$ term. However, when stability analysis is desired, it is beneficial to subtract the $\omega$ term, as will become clear in Section~\ref{ch:KHDAE}. Since no regularity assumptions on $\vA$ are made here, Equation~\eqref{eq:hbm:Rofx} is applicable to DAEs, where $\vA$ is singular, and also to ODEs in standard first-order form, where $\vA = \vI$. For the latter case,
existence and convergence
of the HBM solutions to isolated limit cycles in ODEs have been shown~\cite{Urabe1965}. 

For very simple systems, the map between Fourier coefficients $\vX$ and the corresponding nonlinearity $\cF \vf\left(t, \cF^{-1}(\vX, t)\right)$ can be determined in closed form.
To treat more general classes of nonlinearities, numerical applications of the HBM usually employ a Newton-type iteration procedure to solve Equation~\eqref{eq:hbm:Rofx}.
Given an iterate $\vX$, the residual is determined by first computing the corresponding periodic signal at $L$ sample points using inverse fast Fourier transform (iFFT).
At these sample points, the nonlinearity is evaluated  and then transformed back to the frequency domain using fast Fourier transform (FFT). This method is called the alternating frequency and time (AFT) method~\cite{Cameron1989}. If the number of samples during FFT is too low compared to the bandwith of the nonlinearity, the residual becomes inaccurate due to aliasing effects. For this reason, a weighted-residual approach, where the Fourier coefficients are not computed collectively using FFT but rather by evaluating the Fourier analysis integral of the inner product of the nonlinear term with the corresponding basis function individually for every frequency, was proposed for nonsmooth nonlinearities~\cite{Legrand2024,Hashemi2025}. Due to the significant computational load of the weighted residual method, the AFT is used in all examples below.

 A straightforward computation analogous to the one presented in~\cite{Bayer2024} shows that the derivative of the HBM residual~\eqref{eq:hbm:Rofx}, which is required for the numerical solution of the HBM using gradient-based methods, is
 \begin{align}\label{eq:hillhbm:dRdX:matrix}
   \pd{\vR}{\vX} = \begin{bmatrix}
        \vJ_0 + \ic N \omega \vA & \vJ_{-1} & \dots & \vJ_{-2N}	\\	
		\vdots & & & \vdots\\
		\vJ_{2N} &  \vJ_{2N-1} & \dots & \vJ_0 - \ic N \omega \vA
    \end{bmatrix}.
\end{align}
In the ODE case, i.e., $\vA = \vI$, Equation~\eqref{eq:hillhbm:dRdX:matrix} becomes the well-known truncated Hill matrix~\cite{Lazarus2010}, which can be used for the purpose of stability computation, either through its spectrum or through a matrix exponential via the Koopman-Hill method. The following section generalizes the Koopman-Hill method to DAE cases where $\vA$ may be singular.

\section{Koopman-Hill method for DAEs}\label{ch:KHDAE}

This section presents our main contribution,
the generalization of the Koopman-Hill method for the stability analysis of periodic solutions to the case of non-autonomous differential-algebraic systems of the form~\eqref{eq:Multibody:nonlinDAE}, where $\vA \in \Rspace^{n\times n}$ may be singular and~$\vf:\Rspace\times\Rspace^n \rightarrow \Rspace^n$ is $T$-periodic in its first argument.
This system class includes frictional mechanical multibody systems as shown in Section~\ref{ch:Multibody}.

The Koopman-Hill method for ODEs is an efficient and accurate method to determine the monodromy matrix of a periodic solution through a matrix exponential of the HBM Jacobian~\eqref{eq:hillhbm:dRdX:matrix}, which is also known as the Hill matrix~\cite{Bayer2023,Bayer2024,Bayer2025}.
Below, we show that a similar relationship between the monodromy matrix and the Hill matrix holds for the DAE case, but the matrix exponential is replaced by a generalized fundamental solution matrix, whose determination involves a Drazin inverse.
In essence, the subsequent section presents the derivation of the formula \eqref{eq:abstract:mainFormula} for the monodromy matrix of periodic solutions in DAEs by following the Koopman-based argument of~\cite{Bayer2023} and generalizing where necessary.

\subsection{Floquet theory for ODEs}\label{sec:floquet}

Having determined a periodic solution $\vx^\rp$ through the harmonic balance method, one can examine its stability properties by slightly perturbing the solution.
To this end, a standard procedure for ODEs of the form \eqref{eq:Multibody:nonlinODE} is to linearize the dynamics around the periodic solution and analyze the stability of the origin of the resulting linear time-periodic (LTP) system
\begin{equation}\label{eq:KHDAE:LTPODE}
    \dy = \vJ(t)\vy\textnormal{,}\quad \vJ(t+T) = \vJ(t):= \left.\pd{\vf}{\vx}\right|_{\vx(t)=\vx^\rp(t)}
\end{equation}
 with period $T = \frac{2\pi}{\omega}$.

Due to linearity, there exists a fundamental solution matrix $\vPh(t_2, t_1)$, which maps solution trajectories $\vy(\cdot)$ of~\eqref{eq:KHDAE:LTPODE} along the flow, i.e., $\vy(t_2) = \vPh(t_2, t_1)\vy(t_1)$ for all $\vy(t_1) \in \Rspace^n$~\cite{Nayfeh1995} and $t_1, t_2 \in \Rspace$.
For smooth ODEs, by Liouville's formula \cite{Teschl2012}, the fundamental solution matrix is invertible, which guarantees uniqueness of solutions in both forward and backward time. In particular, the properties 
\begin{subequations}\label{eq:KHDAE:Phiproperties}
\begin{align}
    \vPh(t_1, t_1) &= \vI \\
    \vPh(t_1, t_2) &= \vPh(t_2, t_1)^{-1}
\end{align}
\end{subequations}
hold for all $t_1, t_2 \in \Rspace$.

The stability of the LTP ODE~\eqref{eq:KHDAE:LTPODE} is governed by its Floquet multipliers $\lambda_l\in\Cspace$, $l = 1, \dots, n$.
These are defined to be the eigenvalues of the monodromy matrix $\mo:=\vPh(T, 0)$, which forwards arbitrary initial conditions $\vy_0\in\Rspace^n$ by one period. The Floquet multipliers are nonzero due to the invertibility of the fundamental solution matrix.
If all Floquet multipliers are located inside the unit circle, i.e. 
\begin{equation}\label{eq:KHDAE:MultCondition}
\vert\lambda_l\vert < 1 \textnormal{ for all } l = 1, \dots, n, 
\end{equation} 
then arbitrary initial perturbations decrease from one period to the next. This ensures that the linear time-periodic dynamics~\eqref{eq:KHDAE:LTPODE} (as well as the corresponding periodic solution $\vx^\rp$) are asymptotically stable~\cite{Teschl2012}.

In the semi-simple case, Floquet's Theorem~\cite{Teschl2012} guarantees a special set of $n$ linearly independent solutions which span the solution space of the LTP ODE~\eqref{eq:KHDAE:LTPODE}.
These so-called Floquet form solutions
\begin{equation}\label{eq:KHDAE:FloquetFormSolution}
    \vy_l(t)=\vp_l(t)\ex^{\alpha_l t}\;,
\end{equation}
$l = 1, \dots, n$, consist of a vector-valued, bounded, $T$-periodic trajectory $\vp_l(t) \in \Cspace^n$ and a scalar exponential factor with Floquet exponent $\alpha_l \in \Cspace$, which is related to an associated Floquet multiplier via
\begin{equation}\label{eq:KHDAE:MultExp}
    \lambda_l = \ex^{\alpha_lT}\textnormal{.}
\end{equation}
If the condition
\begin{equation}\label{eq:KHDAE:ExpCondition}
\mathrm{Re}(\alpha_l) < 0 \textnormal{ for all } l = 1, \dots, n, 
\end{equation}
holds, arbitrary trajectories of the LTP ODE \eqref{eq:KHDAE:LTPODE} (which are linear combinations of the Floquet form solutions \eqref{eq:KHDAE:FloquetFormSolution}) decay to zero, which renders the origin asymptotically stable.
Note that this exponent-based condition \eqref{eq:KHDAE:ExpCondition} is equivalent to the multiplier-based condition \eqref{eq:KHDAE:MultCondition} because of \eqref{eq:KHDAE:MultExp}. The Floquet exponents $\alpha_l$ are not unique, but instead appear in groups of equal real part: Shifting the exponent by $\ic k \omega$ with $k \in \Zspace$ and multiplying the periodic part of~\eqref{eq:KHDAE:FloquetFormSolution} by $\ex^{-\ic k \omega t}$ yields the same Floquet form solution $\vy_l$.

\subsection{Floquet theory for DAEs}\label{sec:floquetDAE}
While Floquet theory provides a well-established theoretical basis for the stability analysis of LTP ODEs~\eqref{eq:KHDAE:LTPDAE}, an equivalent theory for linear time-periodic DAEs is less developed~\cite{Lamour1998,Lamour2003}. 
Having revised some aspects of Floquet theory in the ODE case, we will now discuss their generalization to DAEs.

In analogy to \eqref{eq:KHDAE:LTPODE}, an LTP DAE of the form 
\begin{equation}\label{eq:KHDAE:LTPDAE}
    \vA\dy = \vJ(t)\vy\textnormal{,}\quad \vJ(t+T) = \vJ(t):= \left.\pd{\vf}{\vx}\right|_{\vx(t)=\vx^\rp(t)}\textnormal{.}
\end{equation}
arises through a linearization process around a periodic solution $\vx^\rp$ of  \eqref{eq:Multibody:nonlinDAE}.
Assume that Equation~\eqref{eq:KHDAE:LTPDAE} has $r$ linearly independent admissible solutions $\vvph_1, \dots, \vvph_r$ which satisfy at all times the linearized constraint equations that result from the rank deficiency of $\vA$. In other words, solutions of \eqref{eq:KHDAE:LTPDAE} are constrained to an $r$-dimensional time-dependent linear subspace~$\cA(t)\subset \Rspace^n$, i.e.,~$\vy(t)\in\cA(t)$ for all $ t \geq t_0$.
This is, for example, ensured if the DAE is index-1 tractable or index-2 tractable~\cite{Griepentrog1986}, or when it is a Hessenberg system, as is the case when the considered system models a mechanical rigid-body system with smooth bilateral constraints \cite{Lamour2003}. Consider an invertible matrix $\vV(t) \in \Rspace^{n \times n}$, continuous in $t$, whose first~$r$ columns $\vv_1(t), \dots, \vv_r(t)$ span $\cA(t)$. These first $r$ columns of $\vV(t)$ are collected in $\vV_{r}(t) \in \Rspace^{n \times r}$. The matrix $\vV(t)$ is not unique but can always be chosen periodic due to the periodicity of the DAE~\eqref{eq:KHDAE:LTPDAE}. The matrix 
\begin{align}
    \vP(t) = \vV(t) \begin{pmatrix}
        \vI_r & \vzero \\ \vzero & \vzero
    \end{pmatrix} \vV(t)^{-1}
\end{align}
is a projection matrix onto $\cA(t)$. 
By construction, both $\vP(t)$ and $\cA(t)$ are $T$-periodic due to the periodicity of~\eqref{eq:KHDAE:LTPDAE}.

Because the solutions $\vvph_k(t)\in\Rspace^n$, $k = 1, \dots, r$ lie in the admissible subspace $\cA(t)$, they can be expressed as linear combinations of the subspace basis $\vv_1(t), \dots, \vv_r(t)$ via
\begin{align}\label{eq:dae:fundasols}
\left[ \vvph_1(t), \dots \vvph_r(t)\right] = \vV_r(t) \vPs(t)  \;,
\end{align}
with $\vPs(t) \in \Rspace^{r \times r}$ invertible. Any solution of Equation~\eqref{eq:KHDAE:LTPDAE} can be written as a linear combination of the solutions $\vvph_1, \dots, \vvph_r$ expressed through Equation~\eqref{eq:dae:fundasols}, allowing to define the fundamental solution matrix of the DAE as 
\begin{align}\label{eq:dae:fundamat}
    \vPh(t, t_0) := \vV(t) \begin{pmatrix}
        \vPs(t) \vPs(t_0)^{-1} & \vzero \\ \vzero & \vzero
    \end{pmatrix} \vV(t_0)^{-1} \;.
\end{align}
Indeed, by writing the solution $\vy(t)$ as $\vy(t) = \vV_r(t) \vPs(t) \vxi$ for some constant $\vxi \in \Rspace^r$ and multiplying the expression~\eqref{eq:dae:fundamat} for~$\vPh(t, t_0)$ with the initial condition $\vy(t_0) = \vV_r(t_0) \vPs(t_0) \vxi$, we confirm 
\begin{align}
    \vPh(t, t_0) \vy(t_0) = \vV(t) \begin{pmatrix}
        \vPs(t) \vxi & \vzero \\ \vzero & \vzero
    \end{pmatrix} = \vy(t) \;,
\end{align}
which is the desired fundamental solution property. The above definition of the fundamental solution matrix in a linear DAE aligns with that of~\cite[Theorem~3.1]{Lamour2003}, with the assumption of an index-2-tractable DAE somewhat weakened to uniqueness and constant dimension of the solution space. In contrast to the smooth ODE case, the fundamental solution matrix $\vPh(t, t_0)$ defined in Equation~\eqref{eq:dae:fundamat} is not invertible and is only well-defined for $t \geq t_0$. Its initial condition
\begin{align}
    \vP(t_0) := \vPh(t_0, t_0) = \vV(t_0) \begin{pmatrix}
        \vI_{r\times r} & \vzero \\ \vzero & \vzero
    \end{pmatrix} \vV(t_0)^{-1}
\end{align}
is a projection matrix onto the initial solution space $\cA(t_0)$. This has the effect that the fundamental solution matrix can also be multiplied to arbitrary, non-admissible initial conditions $\vy_0 \in \Rspace^n$, which are first projected to the solution space using $\vP(t_0)$ and then evolved in time on the solution space $\cA(t)$.

In analogy to the ODE case, the monodromy matrix $\mo:=\vPh(T, 0)$ can be used to analyze the stability properties of the origin of~\eqref{eq:KHDAE:LTPDAE}. 
The spectrum of $\mo$ now features $r$ nontrivial Floquet multipliers $\lambda_1, \dots, \lambda_r$ as eigenvalues of $\vPs(t) \vPs(t_0)^{-1}$, and additionally $n-r$ eigenvalues which correspond to the non-admissible directions.
We call these additional eigenvalues $\lambda_{r+1}= \dots= \lambda_n = 0$
``projection multipliers'' to indicate that they do not carry stability information but merely represent the projection onto $\cA(t)$.
Hence, only the $r$ non-zero Floquet multipliers $\lambda_1, \dots, \lambda_r$ need to be taken into consideration for stability analysis.
Note that the condition \eqref{eq:KHDAE:MultCondition} for asymptotic stability of the origin still holds, since the projection multipliers are trivially located inside the unit circle.
It can be shown (cf.\ \ref{sec:floquetform}) that every distinct nontrivial eigenvalue $\lambda_1, \dots, \lambda_r$ of $\mo$ has an associated Floquet form solution of the form~\eqref{eq:KHDAE:FloquetFormSolution} with the corresponding Floquet exponent $\alpha_l$ satisfying \eqref{eq:KHDAE:MultExp}. The trivial projection multipliers $\lambda_{r+1}= \dots= \lambda_n=0$ do not generate any further Floquet form solutions as \eqref{eq:KHDAE:MultExp} cannot be solved for a Floquet exponent in that case.

\subsection{Hill eigenvalue problem for DAEs}\label{ch:KHDAE:Hill}

In the ODE case, it is well known that some of the eigenvalues of the truncated Hill matrix \eqref{eq:hillhbm:dRdX:matrix} approximate the Floquet exponents, as proven in~\cite{Zhou2004}.
By appropriately modifying the derivation in~\cite{Lazarus2010}, the Hill eigenvalue problem will be generalized to DAEs in the following subsection.

Substituting one of the $r$ Floquet form solutions \eqref{eq:KHDAE:FloquetFormSolution} into~\eqref{eq:KHDAE:LTPDAE} and multiplying with $\ex^{-\alpha t}$ yields
\begin{equation}
    \vA\dot{\vp}(t) + \vA\alpha\vp(t) = \vJ(t)\vp(t)\textnormal{.}
\end{equation}
After expressing both $\vp$ and $\vJ$ as their complex-valued Fourier series 
\begin{equation}\label{eq:KHDAE:Fourier}
    \vJ(t) = \sum_{l=-\infty}^{\infty} \vJ_l\ex^{\ic l\omega t}, \quad \vp(t) = \sum_{l=-\infty}^{\infty} \vp_l\ex^{\ic l\omega t}
\end{equation}
and applying an index shift, we arrive at
\begin{equation}
	\sum_{k=-\infty}^{\infty}\!\!\left( -\alpha\vA\vp_k-\ic k\omega\vA\vp_k   + \!\!\sum_{l=-\infty}^{\infty} \vJ_l\vp_{k-l}\right)\!\ex^{\ic k\omega t}= 0\textnormal{,}
\end{equation}
which holds for all $t\geq t_0$.
This is a Fourier series where every coefficient vanishes independently.
In particular, the $k$-th summand can be expanded into
\begin{equation}\label{eq:KHDAE:deriveHEP}
	\alpha\vA\vp_k \!\!=\! -\ic k\omega\vA\vp_k + 
	\begin{bmatrix}
		\cdots \!\!\!\!& \vJ_{k+1}\!\! & \vJ_{k}\!\! & \vJ_{k-1} \!\!\!\!&\cdots
	\end{bmatrix}\!\!\!
	\begin{bmatrix}
		\vdots \\ \vp_{-1} \\ \vp_{0} \\ \vp_{1} \\ \vdots
	\end{bmatrix}\!\!\textnormal{.}
\end{equation}
After evaluating this expression for all~$k=..., -1, 0, 1, ...$, the resulting terms are concatenated into
\begin{equation}\label{eq:KHDAE:generalizedHill}
	\alpha \!
	\underbrace{\begin{bmatrix}
			\ddots\!\!&&&&\\
			&\!\!\vA\!\!&&&\\
			&&\!\!\vA\!\!&&\\
			&&&\!\!\vA\!\!&\\
			&&&&\!\!\ddots
	\end{bmatrix}\!}_{\vA_\infty}
	\underbrace{\!\begin{bmatrix}
			\vdots \\ \vp_{-1} \\ \vp_{0} \\ \vp_{1}\\ \vdots
	\end{bmatrix}\!}_{\vp_\infty} \!
	= \vH_\infty
	\underbrace{\!\begin{bmatrix}
			\vdots \\ \vp_{-1} \\ \vp_{0} \\ \vp_{1}\\ \vdots
	\end{bmatrix}\!}_{\vp_\infty}
\end{equation}
with the bi-infinite vector of Fourier coefficients~$\vp_\infty$, the bi-infinite block-diagonal matrix $\vA_\infty$ and the generalized Hill matrix
\begin{equation}\label{eq:KHPDAE:hillmat}
    \vH_\infty = \!\begin{bmatrix}
			\ddots &  \vdots & \vdots & \vdots & \iddots \\
			\cdots  & \!\!\!\vJ_0+\ic \omega\vA\!\!\! & \vJ_{-1} & \vJ_{-2}  & \cdots \\
			\cdots & \vJ_1 & \!\!\!\vJ_0 \!\!\!& \vJ_{-1} &  \cdots \\
			\cdots & \vJ_2 & \vJ_1 & \!\!\!\vJ_0-\ic \omega\vA \!\!\!&  \cdots \\
			\iddots &  \vdots & \vdots & \vdots & \ddots \\
	\end{bmatrix}\textnormal{.}
\end{equation}
In other words, in the context of LTP DAEs, the infinite dimensional Hill eigenvalue problem becomes a generalized eigenvalue problem
\begin{equation} \label{eq:KHPDAE:generalized}
	\alpha\vA_\infty\vp_\infty=\vH_\infty\vp_\infty\textnormal{,}
\end{equation}
which is solved by determining the roots of the characteristic polynomial
\begin{equation}
	\det(\alpha\vA_\infty-\vH_\infty)=0\textnormal{.}
\end{equation}
An eigenvalue~$\alpha$ together with its corresponding eigenvector~$\vp_\infty$ fully characterizes a Floquet form solution of the DAE~\eqref{eq:KHDAE:LTPDAE}. 
Indeed, the classical Hill eigenvalue problem~\cite{Lazarus2010} is recovered for the special case~$\vA=\vI$, yielding purely diagonal $\ic \omega $ terms on the main diagonal of the Hill matrix and an identity operator for $\vA_{\infty}$.

Just as in the ODE case, the truncated problem
\begin{equation} \label{eq:KHDAE:truncatedGeneralizedHill}
	\hat{\alpha}\vA_N\vp_N=\vH_N\vp_N\textnormal{,}
\end{equation}
where the subscript $N$ indicates that only the centermost $[n(2N+1)\times n(2N+1)]$-blocks of $\vA_\infty$ and $\vH_\infty$ are retained, is considered for the sake of numerical tractability. Here, the eigenvalue $\hat{\alpha}$ and the eigenvector~$\vp_N$ are approximations of the Floquet exponent $\alpha$ and the $n(2N+1)$ centermost entries of $\vp_\infty$, respectively.
A comparison between Equations~\eqref{eq:hillhbm:dRdX:matrix} and~\eqref{eq:KHPDAE:hillmat} reveals that the Jacobian of the HBM equations is indeed equal to the truncation of the infinite-dimensional Hill matrix $\vH_{\infty}$.

The error incurred through the truncation does not affect all eigenvalues of the Hill eigenvalue problem equally.
Therefore, classical eigenvalue-based Hill stability methods rely on a sorting procedure to identify the eigenvalues which are best approximating the Floquet exponents. As already described in Sections \ref{sec:floquet} and \ref{sec:floquetDAE}, the Floquet multipliers are expected to occur in groups with equal real part and imaginary parts differing by integer multiples of~$\omega$. Within each group, the eigenvalues that are located closer to the center of the group are expected to be better approximations of the Floquet exponents than the ones located more towards the outer edges of the group, which may become spurious. Both the symmetry-based sorting criterion~\cite{Guillot2020} and the imaginary-part-based criterion~\cite{Zhou2004} therefore attempt to locate eigenvalues which are expected to be correct in this sense.

The smoother $\vf$ is, the faster the Fourier coefficients $\vJ_l$ of $\vJ(t)$ decay with increasing $l$. Therefore, the truncation effects on the Hill matrix are especially pronounced if $\vf$ is nonsmooth, complicating the process of selecting the correct eigenvalues as accurate approximations for the Floquet exponents. This is exemplified by the frictional oscillator example discussed in Section \ref{sec:frictionoscillator}.
For ODEs, the
Koopman-Hill method offers an alternative, sorting-free stability approach. The following subsection generalizes this method to DAEs.

\subsection{Koopman-Hill method for DAEs}\label{ch:KHDAE:LiftedDynamics}
The Koopman-Hill method for ODEs \cite{Bayer2023,Bayer2024,Bayer2025} is a stability analysis tool that avoids a complete eigenvalue decomposition of the truncated Hill matrix and the subsequent sorting
by instead directly calculating the monodromy matrix $\mo$.
In this section, we derive a generalized Koopman-Hill formulation for DAEs by appropriately modifying the steps proposed in~\cite{Bayer2023}. In particular, this includes the use of the
Koopman framework \cite{Brunton2022} to re-interpret the linear time-periodic DAE~\eqref{eq:KHDAE:LTPDAE} as a linear time-invariant DAE of higher dimension. 

Originally developed to extract global information from nonlinear autonomous dynamics via its spectral decomposition, known as Koopman eigenfunctions~\cite{Koopman1931,Mauroy2016,Mauroy2020b}, the Koopman operator has recently been applied mainly in data-driven modeling to construct linear systems that capture nonlinear behavior~\cite{Brunton2022,Bruder2021,Budisic2012}. We provide here only a short overview over selective aspects of the Koopman framework. In particular, the estimation of the Koopman operator from data is not covered. For a broader treatment of the considered methodology, the reader is referred to~\cite{Mauroy2020a, Budisic2012}.

The Koopman framework amounts to a perspective shift. Instead of the evolution of states in the state space, the evolution of functions in a function space is considered. Classically, for nonlinear autonomous dynamical systems, this function space can be an arbitrary Banach space under some mild assumptions~\cite{Mauroy2020}. With the particular aim of deriving the Koopman-Hill method from the LTI DAE~\eqref{eq:KHDAE:LTPDAE}, the suitable function space is the 
infinite dimensional space
\begin{equation}\label{KHDAE:observables}
	\mathfrak{F}=\mathrm{span}\{\mathrm{\vg}_k\mid k\in\mathbb{Z}\}, \quad \mathrm{\vg}_k(t,\vy)=\vy \ex^{-\ic k\omega t}\textnormal{,}
\end{equation}
spanned by products of the state $\vy$ and Fourier basis functions in time, ordered descendingly. The time $t$ is considered here as an extended state with dynamics $\dot{t} = 1$ and initial condition $t(0) = 0$ to make the DAE~\eqref{eq:KHDAE:LTPDAE} autonomous but nonlinear. 

The infinitesimal Koopman generator $\cL: \mathfrak{F} \rightarrow \mathfrak{F}$ \cite{Mauroy2020a} denotes the linear operator which maps any observable  $g \in \mathfrak{F}$ to its Lie derivative~$\dot{g} \in \mathfrak{F}$.
The linearity of the infinitesimal Koopman generator comes at the cost of dealing with a mapping on a (generally infinite-dimensional) function space $\mathfrak{F}$ instead of the (finite-dimensional) state space~$\Rspace^n$. 
A matrix representation of this linear operator, mapping the basis functions to their Lie derivatives, can be found by considering the Lie derivative of each basis function independently. 

In the particular case considered here of the linear time-periodic DAE ~\eqref{eq:KHDAE:LTPDAE}, the Lie derivative of the $k$-th basis function of the form~\eqref{KHDAE:observables} is
\begin{equation}
	\dot{\mathrm{\vg}}_k=-\ic k\omega\vy \ex^{-\ic k\omega t}+\dot{\vy}\ex^{-\ic k\omega t}\textnormal{.}
\end{equation}
By left-multiplying with $\vA$ and inserting the Fourier expanded dynamics \eqref{eq:KHDAE:Fourier} to eliminate $\dot{\vy}$, we obtain
\begin{align}\label{eq:KHDAE:Liederivative}
	\vA\dot{\mathrm{\vg}}_k&=-\ic k\omega\vA\vy \ex^{-\ic k\omega t} +\sum_{l=-\infty}^{\infty}\vJ_l\vy\ex^{-\ic(k-l)\omega t}\nonumber\\
	&=\! -\ic k\omega\vA\vg_k + 
	\left[\begin{matrix}
		\cdots \!\!\!\!& \vJ_{k+1}\!\! & \vJ_{k}\!\! & \vJ_{k-1} \!\!\!\!&\cdots
	\end{matrix}\right]\!\!
	\underbrace{\begin{bmatrix}
		\vdots \\ \vg_{-1} \\ \vg_{0} \\ \vg_{1} \\ \vdots
	\end{bmatrix}}_\vg\!\!\textnormal{.}
\end{align}
After noticing the similarity to \eqref{eq:KHDAE:deriveHEP}, one can proceed as in Section \ref{ch:KHDAE:Hill}.
Evaluating \eqref{eq:KHDAE:Liederivative} for all $k= ...,-1,0,1,...$ and concatenating the terms accordingly yields the linear time-invariant DAE
\begin{equation}\label{eq:KHDAE:liftedDAE}
\vA_\infty\dot{\vg}=\vH_\infty	\vg \textnormal{.}
\end{equation}
In other words, the dynamics of the basis functions are described by an infinite dimensional LTI DAE which shares its overall structure with the generalized Hill eigenvalue problem \eqref{eq:KHDAE:generalizedHill}.

To make the dynamics numerically tractable, we only consider the centermost blocks of the system \eqref{eq:KHDAE:liftedDAE}, yielding the finite-dimensional LTI DAE
\begin{equation}\label{eq:KHDAE:truncatedDAE}
	\vA_N\dot{\vz}(t)=\vH_N\vz(t), \quad \vz(0)=\vz_0\textnormal{.}
\end{equation}
While there is no immediate guarantee that this truncation is justified in the DAE case, the error incurred by this truncation in the ODE case is proven to be bounded and goes to zero for sufficiently smooth dynamics as $N$ is increased \cite{Bayer2024}.

Due to linearity, the truncated problem~\eqref{eq:KHDAE:truncatedDAE}
is solved by a fundamental solution matrix.
For LTI DAEs of the form
\begin{equation}\label{eq:KHDAE:LTIDAE}
	\bar{\vA}\dot{\vz}(t)=\bar{\vH}\vz(t), \quad \vz(0)=\vz_0\textnormal{,}
\end{equation}
with $\bar{\vA}$, $\bar{\vH}$ commutative, a closed-form expression for the fundamental solution matrix is available \cite{Campbell1976}.
This commutativity property can be enforced without loss of generality for arbitrary matrices $\vA_N$, $\vH_N$ by left-multiplying with an inverted matrix pencil \cite{Golub1986}, i.e.,
\begin{subequations}
    \begin{align}
        \Bar{\vA}&:=(a\vA_N-\vH_N)^{-1}\vA_N\;\textnormal{,}\\
        \Bar{\vH}&:=(a\vA_N-\vH_N)^{-1}\vH_N\;\textnormal{.}
    \end{align}
\end{subequations}
Note that left-multiplying the dynamics \eqref{eq:KHDAE:truncatedDAE} with an invertible matrix does not affect the solution space. The scalar $a\in\mathbb{C}$ can be chosen arbitrarily, as long as the pencil is invertible.
For computational reasons, it makes sense to choose
\begin{equation}
	a = \frac{\Vert\vH_N\Vert}{\Vert\vA_N\Vert}
\end{equation}
for any given matrix norm $\Vert\cdot\Vert$ to improve the conditioning of the numerical inversion of the matrix pencil if the matrices $\vH_N$ and $\vA_N$ are of different orders of magnitude.

According to \cite{Campbell1976}, the fundamental solution matrix of \eqref{eq:KHDAE:LTIDAE} is given by
the generalized matrix exponential
\begin{equation}\label{eq:KHDAE:generalizedExp}
	\vE(t,\bar{\vH},\bar{\vA}) = \ex^{\left(\Bar{\vA}^D\Bar{\vH} t\right)}\bar{\vA}^D\bar{\vA}\textnormal{.}
\end{equation}
The operator $(\cdot)^D$ denotes the Drazin inverse, which is a pseudo-inverse for quadratic matrices.
For regular matrices, $\vX^D=\vX^{-1}$ holds.
In general, the operator inverts the regular components of a matrix and sets the singular components to zero, in the sense that the Drazin inverse of a matrix $\vX\in\Rspace^{n\times n}$ in Schur form
\begin{equation}\label{eq:KHDAE:preDrazin}
	\vX = \vT\begin{bmatrix}
		\vR & \mathbf{0} \\ \mathbf{0} & \vN
	\end{bmatrix}\vT^{-1}\textnormal{,}
\end{equation}
with $\vR \in \mathbb{R}^{r\times r}$ regular and $\vN \in \mathbb{R}^{(n-r)\times (n-r)}$ nilpotent is given by
\begin{equation}\label{eq:KHDAE:postDrazin}
	\vX^D = \vT\begin{bmatrix}
		\vR^{-1} & \mathbf{0} \\ \mathbf{0} & \mathbf{0}
	\end{bmatrix}\vT^{-1}\textnormal{.}
\end{equation}

The expression \eqref{eq:KHDAE:generalizedExp} is a DAE-generalization of the matrix exponential for LTI ODEs (recovered in the special case $\bar{\vA}=\vI$).

Having stated a method to solve the lifted LTI dynamics \eqref{eq:KHDAE:truncatedDAE}, we need to relate solution trajectories $\vz(\cdot)$ in the lifted space to solution trajectories $\vy(\cdot)$ of the original linear time-periodic dynamics \eqref{eq:KHDAE:LTPDAE}. The procedure is equivalent to the ODE-formulation of the Koopman-Hill method \cite{Bayer2024}. As $\vz(\cdot)$ approximates trajectories of the observables \eqref{KHDAE:observables}, the initial condition $\vz_0$ must be chosen in accordance, i.e.,
\begin{equation}
    \vz_0 = \vg(0,\vy_0) = \vW_N \vy_0
\end{equation}
with
\begin{equation}
	\vW_N=\begin{bmatrix}
		\vI&\vI&\cdots&\vI&\vI
	\end{bmatrix}^{\T} \in \mathbb{R}^{n(2N+1)\times n} \textnormal{.}
\end{equation}
The centermost observable (which is equal to $\vy$) can be extracted using
\begin{equation}
    \vy(t) = \vg_0(t,\vy(t)) \approx \vC_N \vz(t)
\end{equation}
with
\begin{equation}
	\vC_N=\begin{bmatrix}
		\cdots&\!\!\!\mathbf{0}&\vI&\mathbf{0}\!\!\!&\cdots
	\end{bmatrix} \in \mathbb{R}^{n\times n(2N+1)}\textnormal{.}
\end{equation}
Consequently, we can approximate the solution of the LTP DAE~\eqref{eq:KHDAE:LTPDAE} as
\begin{equation}\label{eq:KHDAE:KHmethod}
		\vy(t)\approx \vC_N \exp\left( \bar{\vA}^D \bar{\vH} T \right) \bar{\vA} ^D \bar{\vA}\vW_N\vy_0
\end{equation}

Note that this solution formula even holds for non-admissible initial conditions $\vy_0\in\Rspace^n$. This is because, the non-invertible factor $\bar{\vA}^D\bar{\vA}$ in the generalized matrix exponential \eqref{eq:KHDAE:generalizedExp} projects on the space of admissible states in the lifted space. Equivalently for the original linear time-periodic dynamics, we obtain an approximation for the projection matrix onto the initial solution space $\cA(0)$
\begin{equation}
    \vP(0) \approx \vC_N \bar{\vA}^D\bar{\vA} \vW_N
\end{equation}
for arbitrary initial conditions $\vy_0\in\Rspace^n$.

From \eqref{eq:KHDAE:KHmethod}, one can extract the full Koopman-Hill formula,
\begin{subequations}\label{eq:abstract:mainFormula}
\begin{equation}
	\mo\approx\moN = \vC_N \exp\left( \bar{\vA}^D \bar{\vH} T \right) \bar{\vA} ^D \bar{\vA}\vW_N
\end{equation}
with
\begin{align}
	\bar{\vA} &= (a\vA_N-\vH_N)^{-1}\vA_N \;, \\
    \bar{\vH} &= (a\vA_N-\vH_N)^{-1}\vH_N \;, \\
    \vC_N &= \begin{bmatrix}
        \cdots&\!\!\!\mathbf{0}&\vI&\mathbf{0}\!\!\!&\cdots
    \end{bmatrix}\in\mathbb{R}^{n\times n(2N+1)} \;, \\
    \vW_N &= \begin{bmatrix}
        \vI&\vI&\cdots&\vI&\vI
    \end{bmatrix}\T \in \mathbb{R}^{n(2N+1) \times n}\;,
\end{align}
\end{subequations}
as an approximation for the monodromy matrix $\mo$ of the linear time-periodic DAE \eqref{eq:KHDAE:LTPDAE}.
While the error bound derived in \cite{Bayer2025} is not immediately applicable for DAEs, the numerical examples in Section~\ref{ch:Examples} show good convergence behavior, i.e., we can expect that
\begin{equation}
	\lim_{N \to \infty} \moN = \mo
\end{equation}
holds.
The spectrum of this matrix consists of Floquet multipliers which can be used to evaluate the stability of the system as well as projection multipliers which are an artifact of restricting the DAE dynamics through constraint equations.
A remarkable property of the Koopman--Hill method for DAEs is that neither the index of the DAE nor the dimension $r$ of the solution space need to be known a priori. Specifically, the approach can, in principle, even be applied when the index of the DAE changes over a period, provided that the matrix pencil of the lifted LTI DAE is regular.

\section{Examples}\label{ch:Examples}
To validate the stability method derived in Section \ref{ch:KHDAE}, we consider two example systems: a nonlinear mathematical pendulum and a nonsmooth frictional oscillator.

\subsection{Forced mathematical pendulum}

The pendulum depicted in Figure \ref{fig:ex:pendulum:sys} consists of a massless rod of length $l$ that constrains the point $M$ of mass $m$ to a circular path around the origin $O$.
The point mass is horizontally forced by a harmonic excitation
\begin{equation}
F(t)=A\sin(\omega t + \theta),
\end{equation}
where $\omega$ denotes the angular frequency.
The system is affected by a dashpot damper with damping constant $d$ as well as gravity $g$.

\begin{figure}[hbt]
	\centering
    \includegraphics{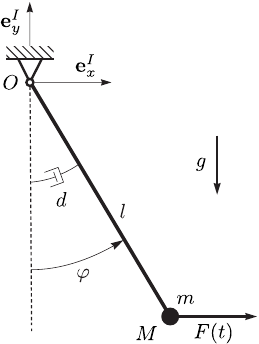}
    \caption{Forced nonlinear mathematical pendulum with one degree of freedom.}
    \label{fig:ex:pendulum:sys}
\end{figure}

The dynamics of the pendulum can be described either as an ODE or as a DAE.
The ODE representation considered here uses the minimal coordinate~$\varphi$.
With the state vector $\vx=\left( \varphi, \dot{\varphi} \right)\T$, the equations of motion (EoMs) in first order form
\begin{equation}\label{eq:Examples:PendulumODE}
	\dot{\vx}= \underbrace{\begin{bmatrix}
			\dot{\varphi} \\
			\frac{1}{ml^2}\left(F(t)l\cos(\varphi)-mgl\sin(\varphi)-d\dot{\varphi}\right)
	\end{bmatrix}}_{\vf(t,\vx)}
\end{equation}
result from the balance of momentum.

Alternatively, the cartesian coordinates $\vq = (q_1, q_2)\T$ with $\vr_{OM} = q_1\ve_x^I + q_2\ve_y^I$ yield the EoMs in DAE form.
Collected in the state vector $\vx = (q_1,q_2,\dot{q}_1,\dot{q}_2,\lambda)^{\T}$ together with
their associated velocities $\dot{\vq}$ and the Lagrange multiplier $\lambda$, the Lagrangian equations of the first kind \eqref{eq:Multibody:Lagrange1} can be derived and be cast in the considered DAE form
\begin{equation}\label{eq:Examples:PendulumDAE}
	\underbrace{\begin{bmatrix}
		1&&&&\\
		&1&&&\\[3pt]
		&&1&&\\[3pt]
		&&&1&\\[3pt]
		&&&&0\\
	\end{bmatrix}}_\vA\!\dot{\vx}\!=\!\underbrace{\begin{bmatrix}
	\dot{q}_1 \\
	\dot{q}_2 \\[3pt]
	\frac{1}{m}\left(2q_1\lambda+F(t)-\frac{d}{l^2}\dot{q}_1\right)\\[3pt]
    \frac{1}{m}\left(2q_2\lambda-mg-\frac{d}{l^2}\dot{q}_2\right)\\[3pt]
	q_1^2+q_2^2-l^2
	\end{bmatrix}}_{\vf(t,\vx)} \;.
\end{equation}
The scalar constraint $g(\vq)=q_1^2+q_2^2-l^2=0$ only depends on the position $\vq$. Because it does not involve $\dot{\vq}$ or $\lambda$, the system \eqref{eq:Examples:PendulumDAE} is an index-3 DAE.

In the following we consider the dimensionless system parameters $m=1$, $l=1$, $g=10$ and $d=0.1$.
 The harmonic excitation force of the system is given by $F(t)=3\sin(t)$ with period $T=2\pi$.
For the purpose of comparison, the classical HBM for system \eqref{eq:Examples:PendulumODE} and the DAE-generalized HBM for system \eqref{eq:Examples:PendulumDAE} are performed with $L=1024$ sample points, harmonic order $N_\mathrm{HBM}=30$ and an absolute tolerance $\mathrm{tol}_\mathrm{abs}=10^{-8}$ of the Newton solver.
As expected, the resulting periodic solutions of the ODE and the DAE coincide.

The stability of the HBM solutions can now be analyzed with the Koopman-Hill method.
In practice, the Koopman-Hill method and the HBM are usually computed with the same harmonic order $N = N_\mathrm{KH} = N_\mathrm{HBM}$.
Here, however, for the purpose of convergence analysis of the Koopman-Hill method for DAEs, we keep the order for harmonic balance $N_\mathrm{HBM}$ constant while considering smaller harmonic orders $N_\mathrm{KH}$ for the Koopman-Hill method.
This ensures that errors are comparable and not influenced by HBM effects.

Since the ODE formulation~\eqref{eq:Examples:PendulumODE} of the pendulum has two states, the linearized system dynamics admits two linearly independent Floquet form solutions. As a consequence, also the admissible solution space~$\cA(t)$ of the corresponding linearization of the DAE~\eqref{eq:Examples:PendulumDAE} is a subspace of the $\Rspace^5$ with constant dimension $r = 2$. Hence, the Floquet theory results of Chapter~\ref{ch:KHDAE} apply directly.

Figure \ref{fig:ex:pendulum:uc} shows the Floquet multipliers $\lambda_{\mathrm{KH}}$ computed for $N_\mathrm{KH}=10$ harmonics.
As it can be observed, both the ODE and DAE formulations produce a complex conjugate pair of Floquet multipliers $\lambda_{1,2}\approx0.24 \pm 0.69$, which coincides visually with the eigenvalues $\lambda_\mathrm{ref}$ of a reference monodromy matrix.
This reference is determined through the classical multiple-pass integration method applied to the ODE formulation, i.e., by time-simulating $n$ linearly independent unit perturbations of the ODE over one period~\cite{Peletan2013}.
For the DAE formulation of the Koopman-Hill method, in addition to the complex conjugate pair of Floquet multipliers, we obtain three projection multipliers $\lambda_{3,4,5}=0$ at the origin.
This is expected, as the subspace of admissible states $\cA(t) \subset \Rspace^5$ of the DAE is only of dimension $r=2$, which is the dimension of the equivalent ODE.

\begin{figure}[hbt]
	\centering
    \includegraphics{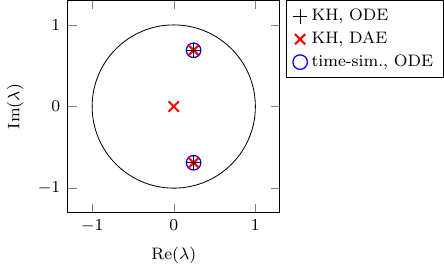}
    \caption{Floquet multipliers of the pendulum system, comparison of the Koopman-Hill method (KH) for the ODE and DAE formulations with a time-simulated reference.}
    \label{fig:ex:pendulum:uc}
\end{figure}

To investigate the convergence of the Koopman-Hill method for the system, Figure \ref{fig:ex:pendulum:conv} shows the distance of the complex conjugate Floquet multipliers $\lambda_\mathrm{KH}$ to the reference for different harmonic orders $N_\mathrm{KH}$.
In general, the approximation improves for increasing $N_\mathrm{KH}$.
However, one can observe that only every other increase of the harmonic order $N_{\mathrm{KH}}$ actually results in an improved numerical value. 
This phenomenon is due to the fact that non-autonomous systems with harmonic forcing typically only excite odd frequency components, an effect that is also exploited by the authors of \cite{Legrand2024}.
Since the time-simulated reference solution is computed with an absolute tolerance of $10^{-6}$, the values of the Floquet multipliers settle at this order of magnitude for both the ODE and DAE case. Indeed, the distance between the projection multipliers and their reference value $0$, which is not affected by integration tolerances, has a similar convergence rate as the other DAE Floquet multipliers but keeps converging beyond the accuracy $10^{-6}$ of the multiple-pass reference.
\begin{figure}[hbt]
	\centering
    \includegraphics{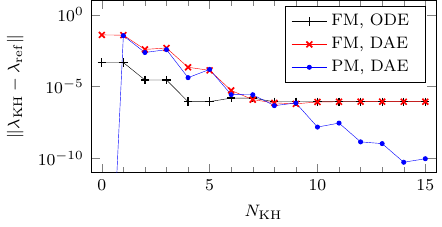}
    \caption{Convergence behavior of the Koopman-Hill method for the pendulum system: Floquet multipliers (FM) of the ODE and DAE formulations and projection multipliers (PM) of the DAE formulation for different harmonic orders $N_\mathrm{KH}$.}
    \label{fig:ex:pendulum:conv}
\end{figure}

The DAE-formulated Koopman-Hill method requires the computation of the Drazin inverse $\bar{\vA}^D$. While many approaches for computing the Drazin inverse of a matrix exist~\cite{Campbell2020}, the approach taken here, as indicated in \eqref{eq:KHDAE:preDrazin} and \eqref{eq:KHDAE:postDrazin}, comes down to inverting a regular matrix block $\vR$ and neglecting a nilpotent matrix block $\vN$ within the Schur decomposition of $\bar{\vA}$.
In the ordered Schur decomposition, both $\vR$ and $\vN$ are of upper triangular structure.
Hence, the main diagonal of $\vR$ consists of all non-zero eigenvalues of $\bar{\vA}$ while all eigenvalues equal to zero appear on the main diagonal of $\vN$. This raises the computational question of how small an eigenvalue must be to be considered zero, i.e., to be sorted into the nilpotent block $\vN$.
Figure \ref{fig:ex:pendulum:drazin} shows the eigenvalues of $\bar{\vA}$ for different harmonic orders $N_\mathrm{KH}$. 
The distribution is clearly divided into two groups.
For all harmonic orders, $\frac{2}{5}$ of the eigenvalues are above a specified tolerance $\epsilon=10^{-4}$ and $\frac{3}{5}$ are below.
This proportion is consistent with the fact that the DAE-formulation of the system produces $2$ Floquet multipliers and $3$ projection multipliers for all harmonic orders $N_{KH}$.
This confirms that the chosen tolerance produces an accurate result.

\begin{figure}[hbt]
	\centering
    \includegraphics{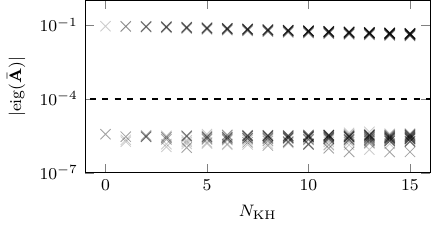}
    \caption{Eigenvalue distribution of $\bar{\vA}$ for the pendulum system, used to compute the Drazin inverse.}
    \label{fig:ex:pendulum:drazin}
\end{figure}

Figure~\ref{fig:ex:pendulum:frc} shows the frequency response curve of the pendulum angle, computed with HBM and subsequent Koopman-Hill stability evaluation for the DAE \eqref{eq:Examples:PendulumDAE}. The angle was computed from the DAE states $q_1$ and $q_2$ via $\tan \varphi(t) = \frac{q_1}{-q_2}$. 
The frequency response curve shows Duffing-like softening behavior.
This behavior is expected: After a Taylor expansion of the trigonometric functions in \eqref{eq:Examples:PendulumODE} and neglecting all terms of order $\cO(\varphi^{4})$, the dynamics in angle coordinates are of the form 
\begin{align}\label{eq:Examples:duffing}
    \ddot{\varphi} + \frac{d}{ml^2} \dot{\varphi} + \frac{g}{l} (\varphi - \frac{1}{6} \varphi^3)  = F(t) \frac{1}{ml} (1 - \frac{1}{2} \varphi^2)\textnormal{.}
\end{align}
Besides the parametric excitation $- \frac{1}{2}F(t)\varphi^2$, Equation \eqref{eq:Examples:duffing} is the Duffing oscillator with negative cubic stiffness, which is well-known for its softening behavior. 
All three expected branches in the region between $\underline{\omega} \approx 2.22$ and $\overline{\omega} \approx 2.94$ are found using HBM and their stability is classified correctly using the Koopman-Hill method for DAEs.

\begin{figure}[hbt]
	\centering
    \includegraphics{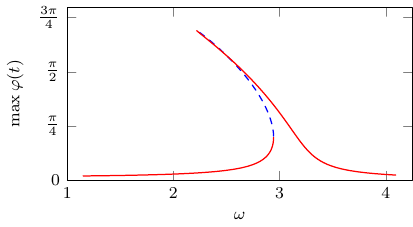}
    \caption{Frequency response curve of the pendulum system in DAE formulation. Stable branches are shown in red, unstable branches in dashed blue.}
    \label{fig:ex:pendulum:frc}
\end{figure}

\subsection{Forced two-body oscillator with friction}\label{sec:frictionoscillator}
As a second example, we consider the frictional oscillator system depicted in Figure \ref{fig:ex:oscillator:sys}.
The same oscillator is used in \cite{Legrand2024} to serve as an example for the nonsmooth frictional HBM formulation.
The system consists of two masses $m_1$ and $m_2$ whose horizontal position is described by the minimal coordinates $\vq = (q_1,q_2)^{\T}$.
Both masses are forced by harmonic excitations
\begin{subequations}
    \begin{align}
	F_1(t)&=A_1\sin(\omega t + \theta_1) \\
	F_2(t)&=A_2\sin(\omega t + \theta_2),
    \end{align}
\end{subequations}
of angular frequency $\omega$. The masses are suspended with two spring-damper units with respective coefficients $k_1$, $k_2$, $d_1$ and $d_2$.

\begin{figure}[hbt]
	\centering
    \includegraphics{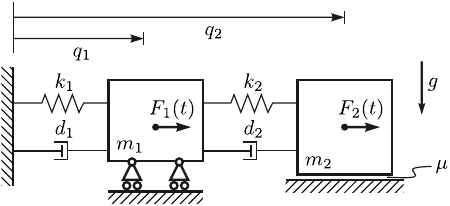}
    \caption{Forced oscillator system with two degrees of freedom, the second mass is affected by Coulomb friction.} 
    \label{fig:ex:oscillator:sys}
\end{figure}

Additionally, the second mass is subjected to a friction force $\lambda_T$ with friction coefficient $\mu$. By setting up the balance of forces for the two-mass oscillator and defining the state vector $\tilde{\vx} = (q_1,q_2,\dot{q}_1,\dot{q}_2)^{\T}$, the EoMs in first order form
\begin{equation}\label{eq:Examples:2DOFoscillatorODE}
	\dot{\tilde{\vx}} \!=\!\!\!\underbrace{\begin{bmatrix}
		0\!\!\!\!\!\!&0\!\!\!\!\!\!&1\!\!\!\!\!\!&0\\
		0\!\!\!\!\!\!&0\!\!\!\!\!\!&0\!\!\!\!\!\!&1\\[3pt]
		-\frac{k_1+k_2}{m_1}\!\!\!\!\!\!&\frac{k_2}{m_1}\!\!\!\!\!\!&-\frac{d_1+d_2}{m_1}\!\!\!\!\!\!&\frac{d_2}{m_1}\\[3pt]
		\frac{k_2}{m_2}\!\!\!\!\!\!&-\frac{k_2}{m_2}\!\!\!\!\!\!&\frac{d_2}{m_2}\!\!\!\!\!\!&-\frac{d_2}{m_2}
	\end{bmatrix} \!\!\tilde{\vx}\! + \!\!\begin{bmatrix}
	0\\0\\[3pt]\frac{1}{m_1}F_1(t)\\[3pt]\frac{1}{m_2}\!(\!F_2(t) \!+ \!\lambda_T\!)\!
	\end{bmatrix}}_{\tilde{\vf}(t,\tilde{\vx})}
\end{equation}
are obtained.
Together with a constitutive law which relates the friction force $\lambda_T$ to the horizontal contact velocity $\gamma_T = \dot{q}_2$, the dynamics of the system are fully described. As introduced in Equation~\eqref{eq:Multibody:FrictionProx},
the nonsmooth Coulomb friction law can be expressed as the implicit proximal point equation
\begin{equation}\label{eq:Examples:FrictionProx}
	0\! =\! \underbrace{\dot{q}_2 \!+\! \min\!\left(0,\lambda_T\!+\!\mu\lambda_N \!-\! \dot{q}_2\!\right) \!+\! \max\!\left(0,\lambda_T\!-\!\mu\lambda_N \!-\! \dot{q}_2\!\right)}_{g(\dot{q}_2, \lambda_T)}
\end{equation}
with normal force $\lambda_N = m_2g$, where the choice $\rho = 1$ was made.
Combining the dynamics \eqref{eq:Examples:2DOFoscillatorODE} with the algebraic force law \eqref{eq:Examples:FrictionProx} yields the DAE representation
\begin{equation}\label{eq:Examples:2DOFoscillatorDAE}
		\underbrace{\begin{bmatrix}
			1&&&&\\
			&1&&&\\
			&&1&&\\
			&&&1&\\
			&&&&0\\
	\end{bmatrix}}_\vA\dot{\vx}=\begin{bmatrix}
		\dot{q}_1\\
		\dot{q}_2\\
		\ddot{q_1}\\
		\ddot{q_2}\\
		0
	\end{bmatrix} = \underbrace{\begin{bmatrix}
		 \vrule\\[6pt]
		 \tilde{\vf}(t,\tilde{\vx}) \\
		 \vrule\\[6pt]
		 g(\dot{q}_2, \lambda_T)
		\end{bmatrix}}_{\vf(t,\vx)}
\end{equation}
of the two-body oscillator with state vector $\vx = (q_1,q_2,\dot{q}_1,\dot{q}_2,\lambda_T)$.

Alternatively, the force law \eqref{eq:Examples:FrictionProx} can be replaced by the smoothed approximation \eqref{eq:Multibody:Inclusionlaw:smoothed} using the hyperbolic tangent with regularization parameter $\alpha$. At the cost of disregarding the nonsmooth character of the system, this provides an explicit relation between $\lambda_T$ and $\dot{q}_2$ which can be directly substituted into \eqref{eq:Examples:2DOFoscillatorODE}. In other words, an ODE describing the dynamics of the two-mass oscillator is obtained. Below, we will use this approximate smoothed system representation as a comparison to the nonsmooth DAE \eqref{eq:Examples:2DOFoscillatorDAE}.

\begin{table}[hbt]
    \centering
    \caption{Parameter sets of Cases A and B considered for the frictional oscillator.}
    \begin{tabular}{lrr}
         \toprule
         Parameter & Case A & Case B \\ 
         \midrule
         $m_1, m_2$ &  $1$ & $1$\\
         $k_1, k_2$ &  $1$ & $1$\\
         $d_1, d_2$ &  $0.02$ & $0.02$\\
         $g$ &  $10$ & $8$\\
         $\mu$ &  $0.9$ & $0.9$\\
         $A_1$ &  $20$ & $20$\\
         $A_2$ &  $10$ & $0$\\
         $\theta_1$ &  $0.4398$ & $\frac{\pi}{2}$\\
         $\theta_2$ &  $2.0106$ & $-$\\
         $\omega$ &  $2\pi$ & $0.293$\\
         \bottomrule         
    \end{tabular}
    \label{tab:Examples:ParametersOscillator}
\end{table}

In the following, we consider two distinct periodic solutions produced by the parameters given in Table \ref{tab:Examples:ParametersOscillator}. Both periodic solutions are determined with the HBM for DAEs and subsequently analyzed with the Koopman-Hill method for DAEs. The parameters in Case A are chosen such that the corresponding periodic solution features a prominent sticking phase and produces Floquet multipliers close to the unit circle. In contrast, the parameters in Case B are chosen in accordance with an example in \cite{Legrand2024} for the purpose of verification and to supplement the results obtained therein with a stability analysis.

In both cases, a classical Newton scheme faces major convergence issues due to the nonsmoothness of the system description, but the trust-region-dogleg solver of the MATLAB function \texttt{fsolve} is able to solve the residual \eqref{eq:hbm:Rofx} of the HBM without numerical issues, as discussed in~\cite{Legrand2024}.
 As a comparison, periodic solutions of the regularized ODE formulation of the oscillator with regularization parameters $\alpha \in [50, 200]$ are determined using the HBM and Koopman-Hill method for ODEs.
Both the classical and the DAE-generalized HBM are performed with $L=4096$ sample points, harmonic order $N_\mathrm{HBM}=100$ and an absolute tolerance of  $\mathrm{tol}_\mathrm{abs}=10^{-8}$.

\begin{figure*}[hbt]
\centering
\captionsetup[subfigure]{justification=centering}
    \begin{subfigure}[t]{0.49\textwidth}
        \raggedleft
        \includegraphics{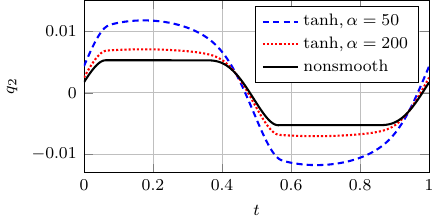}
        \caption{Displacement of the second mass.}
        \label{fig:ex:oscillator:q2}     
    \end{subfigure}
    \hfill
    \begin{subfigure}[t]{0.49\textwidth}
        \raggedleft
        \includegraphics{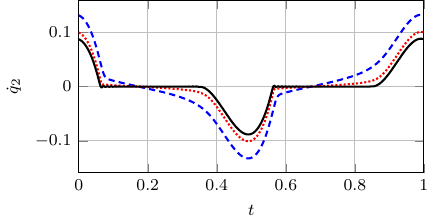}
        \caption{Velocity of the second mass.}
        \label{fig:ex:oscillator:dq2}
    \end{subfigure}
    \hfill
    \begin{subfigure}[b]{0.49\textwidth}
        \raggedleft
        \includegraphics{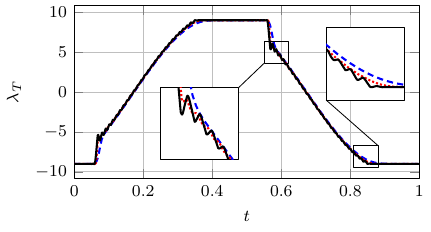}
        \caption{Trajectory of the Coulomb friction force.}
        \label{fig:ex:oscillator:lambda}
    \end{subfigure}
    \hfill
    \begin{subfigure}[b]{0.49\textwidth}
        \raggedleft
        \includegraphics{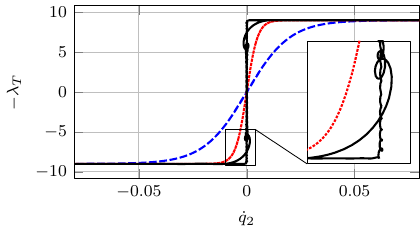}
        \caption{HBM solution of the Coulomb friction force law.}
        \label{fig:ex:oscillator:law}
    \end{subfigure}
    
    \caption{HBM solution in Case A of the nonsmooth oscillator (DAE), comparison to smoothed friction approximation (ODE).}
\end{figure*}

First, the dynamics of the nonsmooth DAE formulation and the smoothed ODE formulation in Case A are considered.
The displacement and velocity of the frictional second mass along the periodic solution are depicted in Figures \ref{fig:ex:oscillator:q2} and \ref{fig:ex:oscillator:dq2}.
The nonsmooth formulation of the oscillator system shows a clear distinction of sticking and slipping phases.
While the periodic solutions of the smoothed system description do not fully capture this distinction, the approximation of the nonsmooth system dynamics improves for increasing~$\alpha$.

The trajectory of the  Coulomb friction force $\lambda_T$ is plotted in Figure \ref{fig:ex:oscillator:lambda}.
While $\lambda_T$ is constant in slipping phases, the behavior in sticking phases is more complex since the force may take a range of values~$\lambda_T\in\left[-\mu\lambda_N,\mu\lambda_N\right] = [-9,9]$.
At the start of each sticking face, the friction force discontinuously jumps to the value necessary to prevent movement, which favors the so-called Gibbs phenomenon.
This phenomenon causes strong oscillations during the whole sticking phase and cannot be reduced by increasing the harmonic order $N$. Although the smoothed dynamics do not suffer from this effect, it can be observed that the steeper approximation $\alpha=200$ also shows a gentle ripple when transitioning from slip into stick.

The dry friction force law in Figure \ref{fig:ex:oscillator:law} is obtained by plotting the coordinates $\left(\dot{q}_2(t_i),-\lambda_T(t_i)\right)$ for all $L=4096$ sample points.
The nonsmooth formulation shows a special behavior: The graph follows different paths when transitioning from slip to stick and vice versa.
In essence, when entering the sticking phase, the graph is curved similar to the smoothed $\tanh$-formulations.
However, on the way back to slipping, the graph of $\left(\dot{q}_2(t_i),-\lambda_T(t_i)\right)$ closely follows the rectilinear set-valued graph of the Coulomb friction law and features knot-like turns.
Again, these properties are a consequence of truncating the Fourier series at a finite harmonic order $N$. 

To compare the accuracy of the different system formulations for the two-body oscillator in Case A, the violation of the constraint equation \eqref{eq:Examples:FrictionProx} over a single period is depicted in Figure \ref{fig:ex:oscillator:constraint}.
For the nonsmooth system formulation, the violation during sticking and slipping phases occurs in different orders of magnitude, reaching its peak right at the slip-to-stick transition before abruptly dropping down.
This matches the observed non-rectilinear and therefore constraint-violating graph in Figure \ref{fig:ex:oscillator:law}.
In contrast, the periodic solutions of the smoothed system mainly lose accuracy during the sticking phase, where the hyperbolic tangent has the largest approximation errors compared to the nonsmooth force law.

\begin{figure}[hbt]
	\centering
    \includegraphics{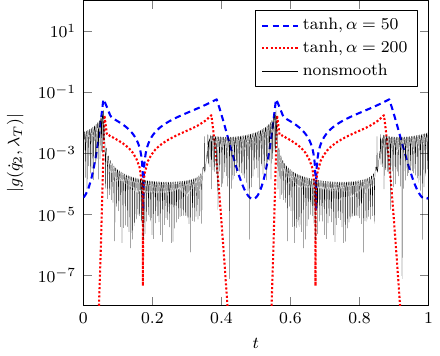}
    \caption{Constraint violation of the nonsmooth and smoothed HBM solutions in Case A.}
    \label{fig:ex:oscillator:constraint}
\end{figure}

To evaluate the stability of these periodic solutions, the Koopman-Hill method is applied to both the nonsmooth DAE formulation and the smoothed ODE formulation with regularization parameter~$\alpha=200$. The results are compared against a highly accurate reference solution determined using a switch model as introduced in \cite{Leine2004}, where the friction force is substituted into the dynamics directly, solving a piecewise linear dynamics exactly during the continuous phases and applying a saltation matrix whenever a stick-slip or slip-stick transition occurs.

\begin{figure}[hbt]
	\centering
    \includegraphics{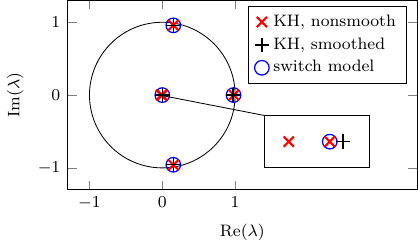}
    \caption{Floquet multipliers in Case A of the oscillator system, determined via the Koopman-Hill method (KH) with harmonic order $N_\mathrm{KH}=30$ for both the nonsmooth DAE and the smoothed ODE $(\alpha = 200)$; comparison to a highly accurate reference determined using the switch model.}
    \label{fig:ex:oscillator:uc}
\end{figure}

The three methods yield four visually coinciding eigenvalues which are plotted in Figure \ref{fig:ex:oscillator:uc}: a complex conjugate pair of Floquet multipliers $\lambda_{1,2}$, a Floquet multiplier $\lambda_3$ close to $+1$ and a 
multiplier $\lambda_4 \approx 0$.
The DAE-formulated Koopman-Hill method features an additional projection multiplier $\lambda_5=0$.
While $\lambda_5$ appears due to the constraint in the $5$-dimensional DAE-formulation, the multiplier $\lambda_4 \approx 0$ is also present for the smoothed system description and in the switch model reference, which are both based on $4$-dimensional ODE-formulations.
This additional multiplier $\lambda_4$ is a result of the nonsmooth frictional nature of the system: 
At the first slip-stick transition ($t \approx $ 0.07), uniqueness in backward time is lost and the dimension of the space $\cA(t)$ of admissible states reduces from $4$ to $3$, even in the ODE case~\cite{Leine2004}. 
This is particularly evident for the reference solution of the switch model, where the fundamental solution matrix is multiplied by a non-invertible saltation matrix at the slip-stick transition~\cite{Leine2004}. Indeed, while $\lambda_4$ is exactly zero for the switch model reference, it admits small but nonzero values of $-0.016$ and $0.005$ for the nonsmooth and regularized HBM solutions, respectively. The multiplier $\lambda_4$ is, therefore, also called a projection multiplier as it is responsible for reducing the subspace dimension from $4$ to $3$. 
In principle, this dimension reduction violates the assumption made in Section~\ref{sec:floquetDAE} about a constant solution space dimension.
However, if the system was directly initialized in a sticking phase, i.e., in a phase of DAE index 2, the subspace $\cA(t)$ would be of constant dimension~$3$ for all~$t\geq0$, suggesting that exactly three Floquet form solutions with corresponding Floquet multipliers exist and the methodology remains applicable.
While the following promising numerical results justify this interpretation and the application of the Koopman-Hill method for frictional systems, further work may be necessary to fully support the theory in the nonsmooth case.

The Floquet multipliers $\lambda_{1,2}$ and $\lambda_3$ produced by the parameters of Case A are located close to the unit circle.
The particular location of the Floquet multiplier $\lambda_3$ can be explained by considering the limit process $\mu \to \infty$ which introduces a bilateral constraint on velocity level between the second mass and the ground.
By locking this degree of freedom, the system description \eqref{eq:Examples:2DOFoscillatorDAE} turns into a DAE of constant index 2. In the limit process, the multiplier $\lambda_3$ converges to $+1$, which corresponds to a perturbation in $q_2$, that is, a perturbation that moves the bilateral constraint to a different position without allowing for any further movement of the second mass.
Additionally, the projection multiplier $\lambda_4$ now corresponds to a non-admissible perturbation in direction of $\dot{q_2}$. This pair of eigenvalues $(\lambda_3, \lambda_4) = (1,0)$ is typical for systems with bilateral constraints on velocity level.

\begin{figure}[hbt]
	\centering
    \includegraphics{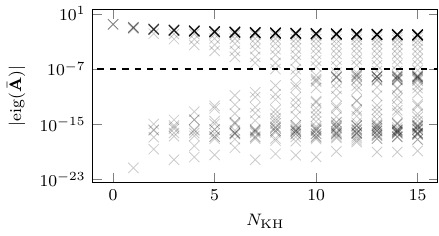}
    \caption{Eigenvalue distribution of $\bar{\vA}$ in Case A of the nonsmooth oscillator system, used to compute the Drazin inverse.}
    \label{fig:ex:oscillator:drazin}
\end{figure}

The computation of the Drazin inverse $\bar{\vA}^D$ for the frictional oscillator is numerically more challenging than for the pendulum system.
The eigenvalue distribution of the matrix $\bar{\vA}$ plotted in Figure~\ref{fig:ex:oscillator:drazin} does not show the same separation of eigenvalues discussed in the context of Figure \ref{fig:ex:pendulum:drazin}.
For small harmonic orders $N_\mathrm{KH}$, $\frac{4}{5}$ of the eigenvalues admit magnitudes above the chosen tolerance $\epsilon=10^{-7}$ and $\frac{1}{5}$ below. For increasing $N_\mathrm{KH}$ these two groups of eigenvalues above and below the threshold approach each other and eventually blend together without a clear distinction at $N_{\mathrm{KH}} \geq 9$. For large $N_{\mathrm{KH}}$, the ratio of eigenvalues above the tolerance is neither $\frac{4}{5}$, as expected for one projection multiplier, nor $\frac{3}{5}$, as expected for two projection multipliers.
Instead, the ratio approaches a value of approximately $75\%$ between these two bounds, indicating the collapse of one solution subspace dimension within the period.
This makes the decision about which eigenvalues of $\bar{\vA}$ to sort into the nilpotent block $\vN$ of \eqref{eq:KHDAE:preDrazin} a numerically badly conditioned problem.
Nonetheless, applying the Koopman-Hill method for these larger harmonic orders still produces comparatively accurate results as discussed in the following convergence analysis.

\begin{figure}[h!]
\centering
\captionsetup[subfigure]{justification=centering}
    \begin{subfigure}[t]{0.49\textwidth}
        \centering
        \includegraphics{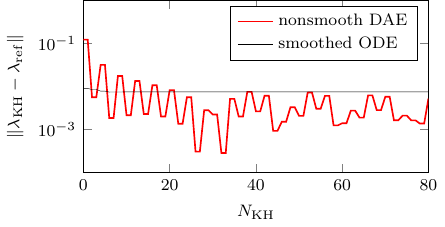}
        \caption{Complex conjugate pair of Floquet multipliers $\lambda_{1,2}$.}
        \label{fig:ex:oscillator:conv12}
    \end{subfigure}
    \hfill
    \vspace{5mm}
    \begin{subfigure}[t]{0.49\textwidth}
        \centering
        \includegraphics{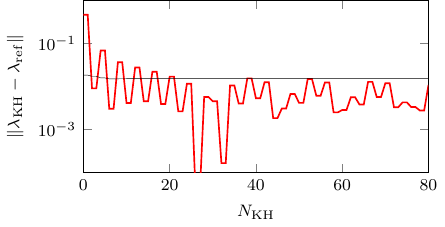}
        \caption{Real-valued Floquet multiplier $\lambda_3$.}
        \label{fig:ex:oscillator:conv3}
    \end{subfigure}
    \hfill
    \vspace{5mm}
    \begin{subfigure}[t]{0.49\textwidth}
        \centering
        \includegraphics{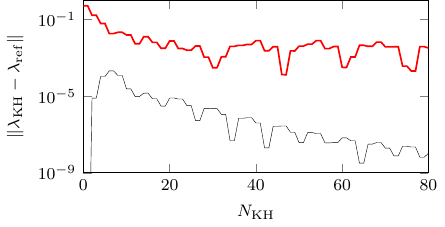}
        \caption{Projection multiplier $\lambda_4$}
        \label{fig:ex:oscillator:conv4}
    \end{subfigure}
    \caption{Convergence behavior of the Floquet and projection multipliers of Case A, determined via the Koopman-Hill method with different harmonic orders $N_\mathrm{KH}$, comparison to highly accurate reference determined using the switch model.}
    \label{fig:ex:oscillator:conv}
\end{figure}

Figure \ref{fig:ex:oscillator:conv} depicts the convergence behavior of the eigenvalues in Case A for different harmonic orders $N_\mathrm{KH}$ with respect to the reference eigenvalues determined with the switch model.
The results of the Koopman-Hill method differ both qualitatively and quantitatively when comparing the nonsmooth DAE-formulation and the smoothed ODE-formulation.
Generally, we cannot expect that the Floquet multipliers converge to the same numerical value due to the differences in the two system descriptions.
While the Floquet multipliers $\lambda_{1,2}$ and $\lambda_3$ in the regularized ODE case converge quickly to a fixed value when increasing $N_\mathrm{KH}$, the multipliers in the DAE case oscillate moderately even for large $N_\mathrm{KH}$.
The oscillations in Figures \ref{fig:ex:oscillator:conv12} and \ref{fig:ex:oscillator:conv3} are of slightly different order of magnitude but are qualitatively almost identical, i.e., both for the complex conjugate pair $\lambda_{1,2}$ and the real-valued multiplier $\lambda_3$, the DAE-formulated Koopman-Hill method gets close to the reference multipliers at the same harmonic orders $N_\mathrm{KH}$.
In particular, one can observe that the distance to the reference Floquet multipliers does not significantly decrease once a certain harmonic order $N_\mathrm{KH}$ is reached.
This effect is likely to be caused by the Drazin inversion procedure discussed in the context of Figure \ref{fig:ex:oscillator:drazin}, which is numerically badly conditioned for large harmonic orders.

While the Floquet multipliers $\lambda_{1,2}$ and $\lambda_3$ of the nonsmooth DAE-formulation match the reference slightly better than the corresponding eigenvalues of the smoothed ODE-formulation, the convergence behavior of the projection multiplier $\lambda_4$ shown in Figure~\ref{fig:ex:oscillator:conv4} is different: While the projection multiplier of the DAE shows oscillations in a similar order of magnitude as the Floquet multipliers, the projection multiplier of the ODE gradually approaches zero. The second projection multiplier $\lambda_5$ that only appears in the DAE-formulation is always identically zero and therefore not shown in Figure \ref{fig:ex:oscillator:conv}.

\begin{figure}[hbt]
	\centering
    \includegraphics{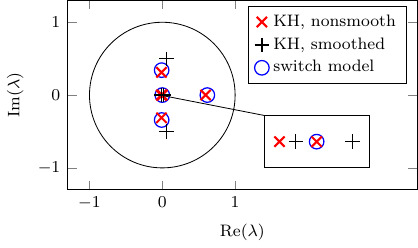}
    \caption{Floquet multipliers in Case B of the oscillator system,
determined via the Koopman-Hill method (KH) with harmonic order $N_\mathrm{KH} = 30$ for the nonsmooth
DAE and the smoothed ODE, comparison to comparison to highly accurate reference determined using the switch model.}
    \label{fig:ex:oscillator:uc_caseB}
\end{figure}

Case B generally exhibits qualitatively similar behavior as Case A. Since the analysis of the HBM solution in Case B is already covered in depth by the authors of \cite{Legrand2024}, the focus of the subsequent discussion lies on the results of the Koopman-Hill method.
Again, the eigenvalues obtained through Koopman-Hill-DAE (nonsmooth case) and Koopman-Hill-ODE (regularized case with~$\alpha=200$) are compared against a highly accurate reference obtained via the switch model.
The resulting Floquet and projection multipliers for the harmonic order $N_\mathrm{KH} = 30$ are shown in Figure \ref{fig:ex:oscillator:uc_caseB}. Similar to Case A, the switch model reference produces a complex conjugate pair of Floquet multipliers $\lambda_{1,2}$, a real-valued positive Floquet multiplier $\lambda_3$ and a projection multiplier $\lambda_4=0$. The Koopman-Hill method for the nonsmooth DAE captures these eigenvalues qualitatively and features an additional projection multiplier $\lambda_5=0$.
The eigenvalues of the regularized ODE, however, are qualitatively different.
Specifically, the real-valued Floquet multiplier $\lambda_3$ cannot be located, instead an additional multiplier close to zero appears. This is potentially due to the stiff nature of the $\mathrm{tanh}$-approximation.

\begin{figure}[h!]
\centering
\captionsetup[subfigure]{justification=centering}
    \begin{subfigure}[t]{0.49\textwidth}
        \centering
        \includegraphics{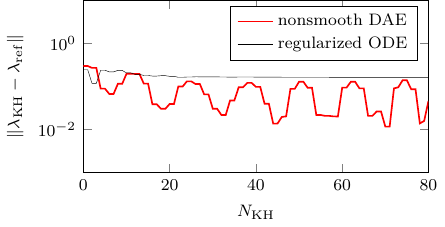}
        \caption{Complex conjugate pair of Floquet multipliers $\lambda_{1,2}$.}
        \label{fig:ex:oscillator:convB12}
    \end{subfigure}
    \hfill
    \vspace{5mm}
    \begin{subfigure}[t]{0.49\textwidth}
        \centering
        \includegraphics{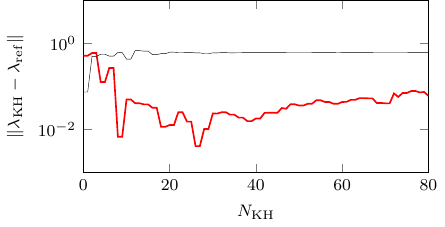}
        \caption{Real-valued Floquet multiplier $\lambda_3$.}
        \label{fig:ex:oscillator:convB3}
    \end{subfigure}
    \hfill
    \vspace{5mm}
    \begin{subfigure}[t]{0.49\textwidth}
        \centering
        \includegraphics{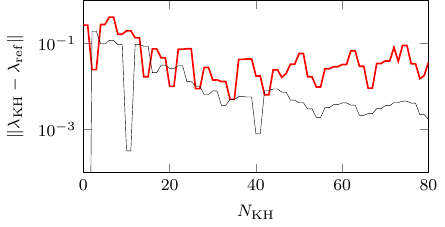}
        \caption{Projection multiplier $\lambda_4$.}
        \label{fig:ex:oscillator:convB4}
    \end{subfigure}
    \caption{Convergence behavior of the Floquet and projection multipliers in Case B, determined via the Koopman-Hill method with different harmonic orders $N_\mathrm{KH}$, comparison to highly accurate reference determined using the switch model.}
    \label{fig:ex:oscillator:convB}
\end{figure}

The convergence behavior of Case B shown in Figure \ref{fig:ex:oscillator:convB} is qualitatively very similar to the properties of Case A discussed in the context of Figure \ref{fig:ex:oscillator:conv}:
While the Floquet multipliers found for the nonsmooth DAE formulation (Figures \ref{fig:ex:oscillator:convB12} and \ref{fig:ex:oscillator:convB3}) oscillate moderately when varying the harmonic order $N_\mathrm{KH}$, the reference Floquet multipliers are generally better approximated than by the regularized ODE formulation.
As already indicated above, this becomes especially clear for the multiplier $\lambda_3$ which is not even qualitatively identified in the ODE case.
Figure \ref{fig:ex:oscillator:convB4} shows that the projection multiplier $\lambda_4$ for the smoothed ODE is, again, located closer to the origin than for the nonsmooth DAE.
While the comparison of Figures \ref{fig:ex:oscillator:conv} and \ref{fig:ex:oscillator:convB} indicates a qualitatively similar convergence behavior, the multipliers determined in Case A are quantitatively closer to the reference than in Case B.
This phenomenon can be explained by the different excitation frequencies
of the two cases.
As discussed in depth by the authors of \cite{Legrand2024}, the dynamics of the frictional oscillator becomes richer if the excitation frequency is close to a resonance frequency of the system, which applies in Case B.

Another potential stability method for DAEs is the determination of the Floquet exponents from the truncated generalized Hill eigenvalue problem~\eqref{eq:KHDAE:truncatedGeneralizedHill} with subsequent sorting, a method that is often applied in ODE contexts.
Figure \ref{fig:ex:oscillator:Hill} shows the spectrum of the generalized Hill eigenvalue problem \eqref{eq:KHDAE:generalizedHill} for Case B of the frictional oscillator ($+$), together with the switch model (\textcolor{blue}{$\circ$}) and the exponents determined from the Floquet multipliers obtained through nonsmooth Koopman-Hill using Equation~\eqref{eq:KHDAE:MultExp} (\textcolor{red}{$\times$}).
Due to the nonsmooth character of the system description, the eigenvalue distribution of the generalized Hill eigenvalue problem is fragmented and does not feature any visually identifiable exponent groups of equal real part, which would be expected in the analysis of smooth dynamics.
As a consequence, the task of distinguishing the ``correct'' Floquet exponents from the vast majority of spurious eigenvalues is practically impossible.
In particular, the eigenvalues inside the fundamental strip , i.e. $- \frac{\omega}{2} < \mathrm{Im}(\alpha) \leq \frac{\omega}{2}$ are not located near the Floquet exponents produced by the switch model.
While the truncated Hill spectrum features eigenvalues located approximately at $\alpha_\mathrm{ref} + \ic k \omega$ for some $k\in\Zspace$, there is no sorting criterion that distinguishes them from the other spurious eigenvalues without prior knowledge of the true solution.

\begin{figure}[hbt]
	\centering
    \includegraphics{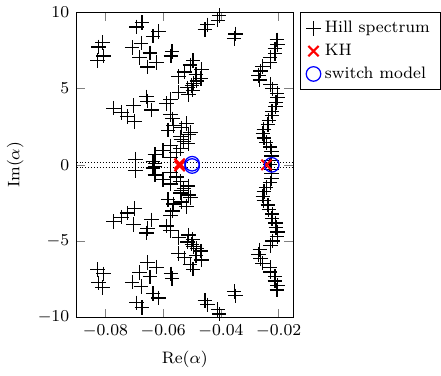}
    \caption{Fragmented Hill spectrum in Case B of the nonsmooth oscillator system, comparison to Koopman-Hill method (KH) and switch model reference, $N_\mathrm{KH}=30$.}
    \label{fig:ex:oscillator:Hill}
\end{figure}

In summary, the Koopman-Hill method for DAEs qualitatively captures the Floquet multipliers of periodic solutions of the frictional oscillator, while its quantitative accuracy is hindered by the Drazin inverse operation executed on the numerically badly conditioned matrix $\bar{\vA}$. Still, the Koopman-Hill method offers a stability analysis conjoined with the nonsmooth HBM-DAE problem, even in situations where the classical Hill eigenvalue method can not be applied due to a highly fragmented, spurious spectrum.

\section{Conclusions}\label{ch:Conclusion}
In this paper, we presented a novel method for the stability analysis of periodic solutions in dynamical systems described by differential algebraic equations.
For this purpose, the Hill eigenvalue problem and the Koopman-Hill framework were generalized to DAEs.
The main ideas of the ODE case remain the same, but crucially, the lifted system becomes a linear time-invariant DAE which necessitates a more sophisticated solution procedure involving a Drazin inverse. The method is, thus, limited by the solvability conditions of linear time-invariant DAEs. In particular, while many algorithms for Drazin inverse computations exist \cite{Campbell2020}, the ill-conditioned problem that may result from the Koopman lift in nonsmooth applications remains numerically challenging.

The proposed Koopman-Hill method is agnostic to the DAE's index. In particular, in contrast to direct time integration of DAEs, no index reduction is necessary in the case of higher-index DAEs. The monodromy matrix, which is generally singular in the DAE case, features nonzero Floquet multipliers but also eigenvalues equal to zero, for which we propose the term ``projection multipliers'' since they do not give rise to a corresponding Floquet form solution, but instead indicate a loss of dimension in the solution space. 

The method is directly applicable to mechanical multibody systems with both smooth and frictional constraints. In the frictional case, however, the solution space reduces in dimension during the first slip-to-stick transition, which leads to an additional projection multiplier. While the numerical results are promising, further research is required to address the challenging question of how the assumption of a constant-dimensional subspace $\cA(t)$ in Chapter \ref{ch:KHDAE} can be suitably generalized to frictional dynamics with loss of uniqueness in backward time.  A detailed answer to this question requires a yet to be conducted in-depth analysis of the solution properties of nonsmooth LTP DAEs with collapsing solution space dimension and their corresponding lifted LTI DAEs with constant solution space dimension, especially with regard to the structure and number of Floquet form solutions. 

By directly computing the monodromy matrix from the generalized Hill matrix, the Koopman-Hill method avoids sorting requirements which are encountered in both the ODE and DAE case for the classical and generalized Hill eigenvalue problem, respectively.
This is particularly useful in the analysis of nonsmooth dynamics, where the Hill matrix typically exhibits an extremely fragmented, spurious spectrum and the number of Floquet exponents may not be obvious a priori.

For the purpose of validation, the Koopman-Hill method for DAEs was applied to a smooth constrained forced pendulum of index 3 and a nonsmooth frictional oscillator of changing index.
Compared to a regularized ODE variant, the nonsmooth formulation better captures the frictional character of the system but has to deal with numerical inaccuracies due to finite truncation of Fourier series which favors the Gibbs phenomenon, an effect that cannot be reduced by increasing the harmonic order of the method.
Our analysis of case B shows that the nonsmooth formulation which is enabled by the Koopman-Hill method can reproduce nonsmooth effects, which are not captured qualitatively by the regularized formulation.

As the first study to generalize the Koopman-Hill method to DAEs, this work was necessarily exploratory in nature, and several theoretical aspects of the method warrant more rigorous treatment in future investigations. In particular, the assumption of an $r$-dimensional solution space can be difficult to justify a priori for a given DAE. Establishing a connection between the index of the lifted matrix pencil and the dimension of the original solution space would both strengthen the theoretical foundations of the method and enhance the reliability of the Drazin inversion procedure. Furthermore, the scope of the present formulation remains limited to systems with constant pseudo-mass matrices, and its extension to the non-constant case $\vA(t,\vx)$ represents a natural and practically significant direction for subsequent research. Additionally, mechanical multibody systems involving spatial Coulomb friction or rigid bodies parameterized by quaternions can generally be described by DAEs and their detailed study using the Koopman-Hill framework remains open for a more detailed future analysis.

\section*{Author contributions, funding disclosure and declaration of competing interest}
No funding was received for conducting this study. The authors have no relevant financial or non-financial interests to disclose. The datasets and code generated and analyzed during the current study are available from the corresponding author on reasonable request. A.S. wrote the initial draft, implemented the software and created the visualizations. A.S. and F.B. performed the formal analysis. All authors contributed to the methodology and reviewed and edited the manuscript. 

\appendix
\section{Floquet form solutions of linear time-periodic DAEs} \label{sec:floquetform}
This section derives Floquet's theorem and the Floquet form solutions for LTP DAEs. All arguments are analogous to the well-known ODE case~\cite{Nayfeh1995}, but care is taken to apply the group property only in forward direction as the fundamental solution matrix is generally not invertible. For the sake of brevity, we assume here $t_0 = 0$, but all arguments hold analogously for arbitrary $t_0 \in \Rspace$.
First, we show Floquet's theorem in the DAE case:

\begin{lemma}[Floquet's theorem for DAEs]\label{eq:Floquetproof:Floquettheorem}
    Consider the linear time-periodic DAE~\eqref{eq:KHDAE:LTPDAE} with constant solution space dimension $r$ and its fundamental solution matrix given in Equation~\eqref{eq:dae:fundamat}. There exists a $T$-periodic matrix $\hat{\vP}(t) \in \Cspace^{r \times r}$ and a constant matrix $\vQ \in \Cspace^{r \times r}$ such that
    \begin{align}
        \vPs(t) \vPs(0)^{-1} = \hat{\vP}(t) \ex^{\vQ t} \;.
    \end{align}
\end{lemma}

\begin{proof}
    Due to the periodicity of the DAE~\eqref{eq:KHDAE:LTPDAE}, the fundamental solution matrix has a $T$-shift property in both arguments simultaneously, i.e., 
    \begin{align}\label{eq:Floquetproof:periodicity}
        \vPh(t, 0) = \vPh(t+T, T), \quad t \geq 0 \;.
    \end{align}
    As discussed above, the group property of the fundamental solution matrix holds in forward time, i.e., 
    \begin{align}\label{eq:Floquetproof:group1}
        \vPh(t + T, 0) = \vPh(t + T, T) \vPh(T, 0), \quad t \geq 0 \;.
    \end{align}
    Combining \eqref{eq:Floquetproof:periodicity} and \eqref{eq:Floquetproof:group1} yields
    \begin{align}\label{eq:FloquetProof:group}
        \vPh(t + T, 0) =  \vPh(t, 0) \mo \;.
    \end{align} 
    Since the matrix $\vPs(t) \vPs(0)^{-1}$ is invertible for all $t \geq 0$, there exists a matrix $\vQ \in \Cspace^{r \times r}$ such that 
    \begin{align}\label{eq:FloquetProof:matrixlog}
        \ex^{\vQ T} = \vPs(T) \vPs(0)^{-1}\;.
    \end{align}
The matrix $\hat{\vP}(t) := \vPs(t) \vPs(0)^{-1} \ex^{-\vQ t}$ is then defined appropriately such that
    \begin{align}\label{eq:FloquetProof:PhiP}
        \vPh(t,0) = \vV(t) \begin{pmatrix}
            \hat{\vP}(t) \ex^{\vQ t} & \vzero\\ \vzero & \vzero
        \end{pmatrix} \vV(0)^{-1}
    \end{align}
    holds. Clearly, $\hat{\vP}(T)$ is the identity matrix. It remains to show that $\hat{\vP}(t)$ is indeed $T$-periodic. To this end, Equation~\eqref{eq:FloquetProof:PhiP} is evaluated at time $t+T$ as
    \begin{align}\label{eq:FloquetProof:PhitplusT}
        \vPh(t + T,0) = \vV(t) \begin{pmatrix}
            \hat{\vP}(t + T) \ex^{\vQ (t+T)}& \vzero\\ \vzero & \vzero
        \end{pmatrix} \vV(0)^{-1}\;.
    \end{align}
    On the other hand, substituting the definitions of $\hat{\vP}(t)$ and $\vQ$ into Equation~\eqref{eq:FloquetProof:group} yields
    \begin{align}\label{eq:FloquetProof:PhitplusT2}
        &\vPh(t + T, 0) =  \vPh(t, 0) \mo \nonumber\\
        =& \vV(t) \begin{pmatrix}
            \hat{\vP}(t) \ex^{\vQ t} & \vzero\\ \vzero & \vzero
        \end{pmatrix} \begin{pmatrix}
            \hat{\vP}(T)\ex^{\vQ T} & \vzero\\ \vzero & \vzero
        \end{pmatrix} \vV(0)^{-1}
    \end{align}
    by making use of the property $\vV(0)^{-1}\vV(T) =\vI$ due to the $T$-periodicity of the matrix $\vV(t)$.
    
    By equating the expressions \eqref{eq:FloquetProof:PhitplusT} and \eqref{eq:FloquetProof:PhitplusT2}, acknowledging the property $\hat{\vP}(T)=\vI$, pre-multiplying with $\vV(t)^{-1}$, and post-multiplying with $\vV(0)$, we obtain
\begin{align}
        \begin{pmatrix}
            \hat{\vP}(t + T) \ex^{\vQ (t+T)} & \vzero\\ \vzero & \vzero
        \end{pmatrix} = \begin{pmatrix}
            \hat{\vP}(t) \ex^{\vQ (t+T)} & \vzero\\ \vzero & \vzero
        \end{pmatrix} \;.
    \end{align}
    With $\ex^{\vQ (t+T)}$ invertible, this directly implies the desired periodicity property of the matrix $\hat{\vP}(t)$.
\end{proof}

In the remainder of this section, we construct the Floquet form solutions. 
\begin{lemma}\label{lem:floquetform}
    Let $\alpha_l$ be an eigenvalue of the matrix $\vQ$ with corresponding eigenvector $\vxi_l$. Then there exists a $T$-periodic vector $\vp_l(t)$ such that the Floquet form solution
    \begin{align}
        \vy_l(t) = \vp_l(t) \ex^{\alpha_l t}
    \end{align}
    is an admissible solution of the DAE~\eqref{eq:KHDAE:LTPDAE} with initial condition
    \begin{align}
        \vy_l(0) = \vV(0) \begin{pmatrix}
            \vxi_l\\ \vzero
        \end{pmatrix}\;\textnormal{.}
    \end{align} 
\end{lemma}
\begin{proof}
    Since $\vxi_l$ is an eigenvalue of $\vB$ with eigenvalue $\alpha_l$, it holds that $\ex^{\vQ t} \vxi_l = \ex^{\alpha_l t} \vxi_l$. This can be seen easily from the power series expansion
    \begin{align}
        \sum_{k=0}^\infty \frac{t^k}{k!} \vQ^k \vxi_l = \sum_{k=0}^\infty \frac{t^k}{k!} \alpha_l^k \vxi_l = \ex^{\alpha_l t} \vxi_l
    \end{align}
    of the matrix exponential.
    Using this property and the fundamental solution matrix for DAEs~\eqref{eq:FloquetProof:PhiP}, the solution $\vy_l(t)$ emanating from the initial condition $\vy_l(0)$ is 
    \begin{align}
        \vy_l(t) &=  \vV(t) \begin{pmatrix}
            \hat{\vP}(t) \ex^{\vQ t} \vxi_l \\ \vzero 
        \end{pmatrix} \nonumber\\
        &=
        \vV(t) \begin{pmatrix}
            \hat{\vP}(t) \vxi_l \\\vzero
        \end{pmatrix} \ex^{\alpha_l t} \;.
    \end{align}
    Thus, the periodic vector $\vp_l(t)$ is given by 
    \begin{align}
        \vp_l(t) = \vV(t) \begin{pmatrix}
            \hat{\vP}(t) \vxi_l \\ \vzero
        \end{pmatrix} \;.
    \end{align}
\end{proof}
When $\vQ$ has $r$ distinct eigenvalues, the solution space is spanned by the $r$ corresponding linearly independent Floquet form solutions through Lemma~\ref{lem:floquetform}. 


\input{main.bbl}
\end{document}

%% file: header.tex
\usepackage[T1]{fontenc}
\usepackage[english]{babel}
\usepackage{lmodern,textcomp,latexsym}
\usepackage{amsthm,amssymb,amsmath,amsfonts,dsfont,mathrsfs,bbm,bm,mathtools}
\usepackage{cite}
\usepackage{color}
\usepackage{booktabs}
\usepackage[footnotesize]{caption}
\usepackage{subcaption}
\usepackage{pgfplots} 
\pgfplotsset{
        compat=newest,
        my axis style/.style={
            xlabel style={
                font=\footnotesize,
            },
            ylabel style={
                font=\footnotesize,
            },
            legend style={
                font=\footnotesize,
            },
            ticklabel style={
                font=\footnotesize,
            },
            every axis plot/.append style={
                line width=1pt,
            },
            legend pos=north west,
            legend columns = 3,
            legend cell align = left,
        },
    }
\usetikzlibrary{calc,positioning}
\newlength\fheight
\newlength\fwidth
\newlength\fheightwide
\newlength\fwidthwide
\newlength\fheightthree
\newlength\fwidththree
\usetikzlibrary{fit}

\usetikzlibrary{external}
\usepackage[title]{appendix}
\usepackage{enumitem}

\theoremstyle{plain}
\newtheorem{theorem}{Theorem}
\newtheorem{lemma}[theorem]{Lemma}

\theoremstyle{definition}

\theoremstyle{remark}

\usepackage[hidelinks]{hyperref}

\newcommand{\diff}[1][]{\mathrm{d}#1}

\newcommand{\pd}[2]{\frac{\partial #1}{\partial #2}}

\newcommand{\td}[2]{\frac{\diff #1}{\diff #2}}

\newcommand{\Sgn}{{\mathop{\mathrm{Sgn}}}}

\newcommand{\mo}{\vPh_T}
\newcommand{\moN}{\vPh_{TN}}

\newcommand{\dy}{\dot{\vy}}

\newcommand{\T}{^{\mathop{\mathrm{T}}}}

\newcommand{\diag}{{\mathop{\mathrm{diag}}}}

\newcommand{\iddots}{{}_{\boldsymbol{\cdot} \vphantom{\frac{1}{1}}} \cdot^{\boldsymbol{\cdot} \vphantom{\frac{1}{1}}}} 

\newcommand{\ic}{\ri} 
\newcommand{\ex}{\re} 

\newcommand{\vzero}{\mathbf{0}}

\newcommand{\Rspace}{\mathbb{R}}
\newcommand{\Cspace}{\mathbb{C}}

\newcommand{\Zspace}{\mathbb{Z}}

\newcommand{\rc}{\mathrm c}

\newcommand{\re}{\mathrm e}

\newcommand{\ri}{\mathrm i}

\newcommand{\rp}{\mathrm p}

\newcommand{\rs}{\mathrm s}

\newcommand{\vla}{\bm{\lambda}}

\newcommand{\vxi}{\bm{\xi}}

\newcommand{\vch}{\bm{\chi}}

\newcommand{\vvph}{\bm{\phi}}

\newcommand{\vPh}{\mathbf \Phi}
\newcommand{\vPs}{\mathbf \Psi}

\newcommand{\ve}{\mathbf e}
\newcommand{\vf}{\mathbf f}
\newcommand{\vg}{\mathbf g}
\newcommand{\vh}{\mathbf h}

\newcommand{\vp}{\mathbf p}
\newcommand{\vq}{\mathbf q}
\newcommand{\vr}{\mathbf r}

\newcommand{\vv}{\mathbf v}

\newcommand{\vx}{\mathbf x}
\newcommand{\vy}{\mathbf y}
\newcommand{\vz}{\mathbf z}

\newcommand{\vA}{\mathbf A}
\newcommand{\vB}{\mathbf B}
\newcommand{\vC}{\mathbf C}
\newcommand{\vD}{\mathbf D}
\newcommand{\vE}{\mathbf E}

\newcommand{\vH}{\mathbf H}
\newcommand{\vI}{\mathbf I}
\newcommand{\vJ}{\mathbf J}

\newcommand{\vM}{\mathbf M}
\newcommand{\vN}{\mathbf N}

\newcommand{\vP}{\mathbf P}
\newcommand{\vQ}{\mathbf Q}
\newcommand{\vR}{\mathbf R}

\newcommand{\vT}{\mathbf T}

\newcommand{\vV}{\mathbf V}
\newcommand{\vW}{\mathbf W}
\newcommand{\vX}{\mathbf X}

\newcommand{\cA}{\mathcal{A}}

\newcommand{\cF}{\mathcal{F}}

\newcommand{\cL}{\mathcal{L}}

\newcommand{\cO}{\mathcal{O}}

